\documentclass[notitlepage]{revtex4-1}
\usepackage{amsthm}
\usepackage{lmodern}
\usepackage[english]{babel}
\usepackage{amsmath}
\usepackage{amssymb}
\usepackage{mathptmx} 
\usepackage{color}
\usepackage[babel]{microtype}
\usepackage{graphicx}
\usepackage{caption}
\usepackage[caption=false]{subfig}
\usepackage{csquotes}
\usepackage{float}
\usepackage[all,cmtip]{xy}
\usepackage{bm}
\usepackage[a4paper, left=3cm, right=2cm]{geometry}
\usepackage{enumitem}

\newtheorem{defi}{Definition}[section]
\newtheorem{thm}[defi]{Theorem}

\newtheorem{lemma}[defi]{Lemma}
\newtheorem{cor}[defi]{Corollary}

\newcommand{\re}{\operatorname{Re}}
\newcommand{\im}{\operatorname{Im}}
\newcommand{\rmi}{\operatorname{i}}
\newcommand{\rme}{\operatorname{e}}

\newcommand{\lcm}{\operatorname{lcm}}

\begin{document}

\title{CONSTRUCTING A POLYNOMIAL WHOSE NODAL SET\\
IS ANY PRESCRIBED KNOT OR LINK}

\author{
B.~Bode\footnote{benjamin.bode@bristol.ac.uk} and M.~R.~Dennis\footnote{mark.dennis@bristol.ac.uk}}

\address{H H Wills Physics Laboratory, University of Bristol, Bristol BS8 1TL, UK
}


\begin{abstract}We describe an algorithm that for every given braid $B$ explicitly constructs a function $f:\mathbb{C}^{2}\rightarrow\mathbb{C}$ such that $f$ is a polynomial in $u$, $v$ and $\overline{v}$ and the zero level set of $f$ on the unit three-sphere is the closure of $B$. The nature of this construction allows us to prove certain properties of the constructed polynomials. In particular, we provide bounds on the degree of $f$ in terms of braid data.
\end{abstract}

\maketitle

\section{Introduction}
\label{intro}

There are various settings in which links can arise in zero level sets of polynomials. Among these the most prominent one is the study of links of isolated singularities \cite{milnor:1968singular}. While only a very restricted class of links can be the link of an isolated singularity, for some types of polynomials every link is possible. For example for any link $L$ there exists a polynomial $f:\mathbb{R}^3\rightarrow\mathbb{C}$ in three real variables with complex coefficients such that $f^{-1}(0)$ is $L$. Even for these types explicit constructions are known for only very few examples.  
Dealing with explicit functions can be advantageous for studying some aspects of knot theory more closely. The critical points of the argument $\arg(f):S^{3}\backslash L\rightarrow S^{1}$ for example relate to circle-valued Morse and Novikov theory, which then connects to Reidemeister torsion and Seiberg-Witten invariants \cite{hl:1999circle}. Explicit polynomials are also necessary for applications in the construction of knotted field configurations in physics.
The goal of this paper is to describe a construction of such polynomials that works for any link.

Knots as vanishing sets of polynomials have been mostly studied in the context of isolated singularities of a polynomial $\mathbb{C}^{2}\rightarrow\mathbb{C}$. For all small $\varepsilon>0$ the intersection of the zero level set with the three-sphere of radius $\varepsilon$ around the singularity is ambient isotopic to the same link, the link of the singularity \cite{milnor:1968singular}. 
For the links which arise as links of isolated singularities of complex polynomials, certain iterated cables of torus links \cite{brauner:1928geometrie, burau:1933knoten, burau:1934verkettungen, le:1972noeuds, kahler:1929verzweigung, zariski:1932topology}, one can explicitly write down the corresponding polynomial. However, this procedure only covers a very restricted class of links. For a more detailed account of this topic we point the reader to \cite{milnor:1968singular, en:1985three}.

Similar considerations apply to polynomials $\mathbb{R}^{4}\rightarrow\mathbb{C}$ with (weakly) isolated singularities \cite{ak:1981all, milnor:1968singular}, but while some knots like the figure-eight knot have been explicitly constructed in this setting \cite{perron:1982noeud, rudolph:1987isolated, looijenga:1971note}, a constructive method that allows to generate polynomials with a zero level set of any given link explicitly is still missing.
Akbulut and King showed in \cite{ak:1981all} that all links arise as links of weakly isolated singularities of real polynomials. However, their proof does not offer an explicit construction of these functions either.

Knots have also been studied as real algebraic curves in $\mathbb{RP}^{3}$ in an area that is now known as real algebraic knot theory. The knots that arise in this way for polynomials of degree less than or equal to five or equal to six with genus at most one have been classified \cite{bjorklund:2011real,mo:2016real}.

It follows from results in real algebraic geometry \cite{bcr:1998real,ak:1981all} that for any link $L$ there exists a polynomial in three real variables with complex coefficients, whose vanishing set is $L$, but the proofs do not come with a method of finding these functions without using algebraic approximation, which is highly impractical in general. In particular, the degree of these polynomials could be arbitratily large.
The functions constructed by Brauner, Perron and Rudolph can be used to generate knots as such algebraic sets in $\mathbb{R}^{3}$ simply by applying a stereographic projection and multypling by a nowhere-vanishing common denominator. We are confronted with a similar situation as in the case of the result by Akbulut and King. While the existence of the polynomials on $\mathbb{R}^3$ vanishing on $L$ is known for any link $L$, explicit constructions are known for only very few of them. In this paper we describe a way of constructing such polynomials that works for any link.  

In fact, as in the previous constructions, we create a function which is defined on four-dimensional real space and is known to have the desired nodal set on $S^{3}$ and apply a stereographic projection to it. In order to do this, an isolated singularity at the origin is not necessary. All that is required is that the intersection of the zero level set of the function, which is defined on $\mathbb{R}^{4}$ (or $\mathbb{C}^{2}$ or $\mathbb{C}\times\mathbb{R}^{2}$), with the unit three-sphere is the desired knot or link.   

Rudolph introduced the notion of transverse $\mathbb{C}$-links \cite{rudolph:2005knot}. These are the links that arise as transverse intersections of zero level sets of complex polynomials $f:\mathbb{C}^{2}\rightarrow\mathbb{C}$ with the unit three-sphere. These functions do not necessarily have an isolated singularity at the origin and the zero level set on three-spheres of small radii does not have to be that same link.
The set of transverse $\mathbb{C}$-links is the same as the set of quasipositive links, which are closures of braids that are products of conjugates of positive Artin generators $\sigma_{i}$ of the braid group \cite{rudolph:2005knot, bo:2001quasi}. 

We weaken the condition on the functions that Rudolph studied to the effect that we are interested in complex-valued polynomials in complex variables $u$, $v$ and $\overline{v}$. We call these polynomials \emph{semiholomorphic}, since they are holomorphic in $u$, but not in $v$.

In \cite{us:2016lemniscate, dkjop:2010isolated} we described a construction of semiholomorphic polynomials with zero level sets on $S^{3}$ in the form of lemniscate knots, a family of knots generalising the family of torus links. The main ingredient of this procedure is that lemniscate knots are closures of braids that have a particularly simple parametrisation in terms of trigonometric functions.
In \cite{us:2016fivetwo} we showed using the example of the knot $5_{2}$ that by using geometric considerations one can find a similar trigonometric parametrisation for a braid even if its closure is not a lemniscate knot. However, this approach is not algorithmic and very time consuming.

In this paper we show that the idea behind the construction of lemniscate knots can be extended to all links.
\begin{thm}
\label{thmintro}
For every link $L$ there is a function $f:\mathbb{C}^{2}\rightarrow\mathbb{C}$ such that $f^{-1}(0)\cap S^{3}=L$ and $f$ is a polynomial in complex variables $u$, $v$ and $\overline{v}$, i.e. $f$ is a semiholomorphic polynomial.
\end{thm}

The proof provides an algorithm that for any given braid $B$ constructs a semiholomorphic polynomial with the closure of $B$ as its zero level set on the unit-three sphere. Note that stereographic projection from $S^{3}$ to $\mathbb{R}^{3}$ then allows us to turn the constructed polynomial $f:\mathbb{C}^{2}\rightarrow\mathbb{C}$ into a polynomial $\tilde{f}:\mathbb{R}^{3}\rightarrow\mathbb{C}$ with the desired shape of its nodal lines.

Having an explicit formalism to generate polynomials comes with the advantage of being able to show certain properties about them. We can for example give upper and lower bounds on the degree of the resulting polynomial $f$ in terms of braid data, the number of strands in each link component and the number of crossings between them. This result carries over to bounds on the degree of $\tilde{f}$, improving results associated to the Nash-Tognoli theorem \cite{bcr:1998real}, which do not contain any information about the degree. We prove the following theorem.
\begin{thm}
\label{thmintro2}
Let $B$ be a braid on $s$ strands of length $\ell$ and let $L$ denote its closure. Then there exists a function $f:\mathbb{C}^{2}\rightarrow\mathbb{C}$ such that
\begin{itemize}
\item $f^{-1}(0)\cap S^{3}=L$,
\item $f$ is a polynomial in $u$, $v$ and $\overline{v}$, i.e. it is semiholomorphic,
\item $f$ is harmonic, i.e. $(\partial_{u}\partial_{\overline{u}}+\partial_{v}\partial_{\overline{v}})f=0$,
\item the intersection of $f^{-1}(0)$ and $S^{3}$ is transverse and $0$ is a regular value of $f|_{S^{3}}$,
\item $\max\left\{s,c_{1}\right\}\leq \deg(f)\leq\max\left\{s,c_{2}\right\}$,
\end{itemize} 
where $c_{1}$ and $c_{2}$ are expressions involving the number of strands in each link component and the number of crossings between them and are defined in Eq. (\ref{eq:lower}) and Eq. (\ref{eq:upper}) respectively.
\end{thm}

The constructed polynomials also find applications in the study of a wide range of physical systems, including quantum-mechanical wavefunctions \cite{pob:2012vortex}, complex scalar optical fields \cite{dkjop:2010isolated}, nematic liquid crystals \cite{ma:2014knotted} and the Skyrme-Faddeev model \cite{sutcliffe:2007knots}, where they are used as initial configurations. Our algorithm hence provides a way of creating a knotted initial condition for any link type in any of these systems.

The structure of this paper is as follows. Section \ref{main} introduces the concept of Fourier parametrisations of braids, which are essential for the construction. We then show how we can use such a parametrisation to construct a semiholomorphic polynomial with a zero level set of the desired form.
In Section \ref{algo} we discuss a method of finding a Fourier braid parametrisation for any braid, making our construction fully algorithmic. In Section \ref{Props} we use the nature of our construction to prove, among other things, bounds for the degree of the resulting polynomials. We also givbe an upper bound for the Morse-Novikov number of a link and show that if the desired link is the closure of a strictly homogeneous braid, then there exists a semiholomorphic polynomial $f$ such that the argument $\arg(f)$ of the restriction of $f$ to the unit three-sphere is a fibration over $S^{1}$.

\section{Semiholomorphic polynomials with knotted nodal sets}
\label{main}
The idea behind the construction of the polynomials is very much in the spirit of \cite{us:2016lemniscate} and \cite{us:2016fivetwo}, but is not restricted to a certain class of knots and is fully algorithmic. We use certain trigonometric parametrisations of a given braid $B$ to define a 1-parameter family of functions from $\mathbb{C}\times[0,2\pi]$ to $\mathbb{C}$ whose zero level sets are $B$. Identifying the 2-periodic variable with the argument of a complex number $v$ allows us to close the braid and define a 1-parameter family of functions whose zero level set is the closure of $B$ for appropriate choices of the parameter. The parametrisation of the braid is chosen in such a way that the resulting function is a polynomial in $u$, $v$ and $\overline{v}$.  

We will think of a braid $B$ on $s$ strands as the union of $s$ disjoint parametric curves in $\mathbb{R}^{3}$, parametrised by their $z$-coordinate between 0 and $2\pi$.
This means that each strand is given by 
\begin{equation}
(X_{j}(t),Y_{j}(t),t),\quad j=0,1,\ldots,s-1,
\label{eq:para}
\end{equation}
where each $X_{j}$ and each $Y_{j}$ is a smooth function $[0,2\pi]\rightarrow\mathbb{R}$.

Additionally, we demand that there is a permutation $\pi_{B}\in S_{s}$, such that $X_{j}(0)=X_{\pi_{B}(j)}(2\pi)$ and $Y_{j}(0)=Y_{\pi_{B}(j)}(2\pi)$ for all $j\in\{1,2,\ldots,s\}$. Given these equalities, we can identify the $z=0$ plane and the $z=2\pi$ plane, which closes the braid to a link $L$. Note that there is a one to one correspondence between the cycles in the cycle notation of $\pi_{B}$ and the components of the link $L$. Furthermore, the length of each cycle is equal to the number of strands making up the corresponding link component.

Projecting the strands $(X_{j}(t),Y_{j}(t),t)$ on the $y=0$ plane will typically result in a braid diagram that allows to read off the braid word of $B$, that is the element of the braid group on $s$ strands $B_{s}$ corresponding to $B$ in terms of the Artin generators $\sigma_{j}$.

We want to describe an algorithm that for any given link $L$ constructs a semiholomorphic polynomial $f:\mathbb{C}\times\mathbb{R}^{2}\rightarrow\mathbb{C}$ such that $f^{-1}(0)\cap S^{3}=L$. This algorithm is a generalisation of the methods used in \cite{perron:1982noeud}, \cite{dkjop:2010isolated} and \cite{us:2016lemniscate}. Let $L$ be the link we want to construct and $B$ be a braid on $s$ strands that closes to $L$.

The first step of the algorithm is to find a parametrisation of the braid $B$ as in Eq. (\ref{eq:para}).
Then we can define a family of functions $g_{a,b}:\mathbb{C}\times[0,2\pi]\to\mathbb{C}$,
\begin{equation}
g(u,t)=\prod_{j=0}^{s-1}(u-aX_{j}(t)-ibY_{j}(t))
\label{eq:braidpoly}
\end{equation}
which has the braid $B$ as its zero level set $g_{a,b}^{-1}(0)$ for any $a>0$, $b>0$.

Every $g_{a,b}$ is by definition a polynomial in the complex variable $u$, but its dependence on $t$ is determined by the parametrisation. In Section \ref{fourier} we show that if the braid parametrisation is of a certain form, $g_{a,b}$ is a polynomial in $u$, $\rme^{\rmi t}$ and $\rme^{-\rmi t}$. As such it is the restriction of a semiholomorphic polynomial $f_{a,b}:\mathbb{C}^2\rightarrow\mathbb{C}$ to $\mathbb{C}\times S^1$.
Note that $f_{a,b}$ can be obtained from $g_{a,b}$ simply by replacing every instance of $\rme^{\rmi t}$ by a complex variable $v$ and $\rme^{-\rmi t}$ by $\overline{v}$, the complex conjugate of $v$.
Section \ref {proof} shows that for small $a$ and $b$ the zero level set of $f_{a,b}$ on the unit three-sphere is ambient isotopic to $L$, i.e. $f_{a,b}^{-1}(0)\cap S^{3}=L$.

\subsection{Fourier Braids}
\label{fourier}

In this section we describe a braid parametrisation as in Eq. (\ref{eq:para}) of a certain form that guarantees that $g_{a,b}$ defined by Eq. (\ref{eq:braidpoly}) can be written as a polynomial in $u$, $\rme^{\rmi t}$ and $\rme^{-\rmi t}$ for all $a$ and $b$.

Recall that we can associate to every braid $B$ with $s$ strands an element $\pi_{B}$ of the symmetric group $S_{s}$ on $s$ elements.
The cycles of $\pi_{B}$ correspond to the link components of the closure of $B$. We will denote the set of cycles of $\pi_{B}$, or equivalently the set of components of $L$, by $\mathcal{C}$. For a given cycle $C$ let $s_{C}$ denote the length of $C$ or equivalently the number of strands that form the link component $C$.

The condition that $X_{j}(0)=X_{\pi_{B}(j)}(2\pi)$ and $Y_{j}(0)=Y_{\pi_{B}(j)}(2\pi)$ for all $j$ ensures that the projection of the braid on the $xy$-plane is a collection of closed curves, one for each link component. Thus for each link component $C$ there exist $2\pi$-periodic, continuous  real functions $F_{C}$ and $G_{C}$ such that for every strand $j$ we have $X_{j}(t)=F_{C}\left(\frac{t+2\pi k}{s_{C}}\right)$, $Y_{j}(t)=G_{C}\left(\frac{t+2\pi k}{s_{C}}\right)$, $X_{\pi_{B}(j)}(t)=F_{C}\left(\frac{t+2\pi (k+1)}{s_{C}}\right)$ and $Y_{\pi_{B}(j)}(t)=G_{C}\left(\frac{t+2\pi (k+1)}{s_{C}}\right)$ for some $(C,k)$.

Let $F_{C}$ and $G_{C}$ be trigonometric polynomials, i.e. 
\begin{equation}F_{C}(t)=\sum_{k=-N_{C}}^{N_{C}}a_{C,k}\rme^{\rmi kt}\text{  and  }G_{C}(t)=\sum_{k=-M_{C}}^{M_{C}}b_{C,k}\rme^{\rmi kt}\end{equation}
with $a_{C,-k}=\overline{a}_{C,k}$ and $b_{C,-k}=\overline{b}_{C,k}$ for all $C\in\mathcal{C}$ and all $k$.
Suppose 
\begin{equation}
\bigcup_{C\in\mathcal{C}}\bigcup_{j=0}^{s_{C}-1}\left(F_{C}\left(\frac{t+2\pi j}{s_{C}}\right),G_{C}\left(\frac{t+2\pi j}{s_{C}}\right),t\right),\quad t\in[0,2\pi]
\label{eq:fourierpara}
\end{equation} 
is a parametrisation of the braid $B$ as in Eq. (\ref{eq:para}). Then as in Eq. (\ref{eq:braidpoly}) this parametrisation with $X_{j}(t)=F_{C}\left(\frac{t+2\pi k}{s_{C}}\right)$ and $Y_{j}(t)=G_{C}\left(\frac{t+2\pi k}{s_{C}}\right)$ for some $(C,k)$ leads to a function $g_{a,b}:\mathbb{C}\times [0,2\pi]\rightarrow\mathbb{C}$ with $B$ as its zero level set for all $a>0$, $b>0$ and we claim it is in fact a polynomial.

\begin{lemma}
\label{poly} The function $g_{a,b}$ can be written as a polynomial in $u$, $\rme^{\rmi t}$ and $\rme^{-\rmi t}$.
\end{lemma}
\begin{proof} 
Obviously all exponents of $u$ are natural numbers, so we only have to show that after expanding the product in Eq. (\ref{eq:braidpoly}) all exponents of $\rme^{\rmi t}$ are integers.

We are going to show this for knots, so there is only one cycle in $\pi_{B}$. Since $g_{a,b}$ is the product over the functions associated with each cycle, this  will show that $g_{a,b}$ is a product of polynomials and hence a polynomial itself for any braid.

Since $g_{a,b}$ has exactly $s$ factors, every term of $g_{a,b}$ after expanding the product consists of exactly $s$ terms which can each be $u$ or $c_{m}\rme^{\frac{\rmi m(t+2\pi j)}{s}}$ for some non-zero integer $m$ and $c_{m}\in\mathbb{C}$. Hence every summand has the form 
\begin{equation}
u^{s-k}T_{j_{1}\ldots,j_{k}}^{m_{1},\ldots,m_{k}}=u^{s-k}\prod_{p=1}^{k}c_{m_{p}}\rme^{\frac{\rmi m_{p}(t+2\pi j_{p})}{s}}
\label{eq:term}
\end{equation}
with $j_{p}\neq j_{r}$ if $p\neq r$.

Moreover, if $T_{j_{1},\ldots,j_{k}}^{m_{1},\ldots,m_{k}}$ appears with the factor $u^{s-k}$, so do all $T_{\tilde{j}_{1},\ldots,\tilde{j_{k}}}^{m_{1},\ldots,m_{k}}$ with $\tilde{j_{i}}\in\{0,\ldots,s-1\},\ \tilde{j}_{i}\neq\tilde{j}_{r}$ if $i\neq r$. Note that if $m_{i}=m_{w}$, then $T_{j_{1},\ldots,j_{i},\ldots,j_{w},\ldots,j_{k}}^{m_{1}\ldots,m_{i},\ldots,m_{w},\ldots,m_{k}}=T_{j_{1},\ldots,j_{w},\ldots,j_{i},\ldots,j_{k}}^{m_{1}\ldots,m_{i},\ldots,m_{w},\ldots,m_{k}}$. This turns the proof of the lemma into a combinatorial problem. Summing over all summands with the same $m_{i}$ becomes summing over all possibilities of choosing $p$ distinct values $j_{p}$ between $1$ and $s$.

We will use induction on $k=s-l$, where $l$ is the exponent of $u$ in the relevant term, to show that all terms $T_{j_{1},\ldots,j_{k}}^{m_{1},\ldots,m_{k}}$ with $s\nmid\sum_{i=1}^{k}m_{i}$ cancel each other, i.e. 
$\underset{j_{i} \text{ disjoint}}{\underset{\in(\mathbb{Z}/s\mathbb{Z})^{k}}{\sum_{(j_{1},j_{2},..,j_{k})}}}
T_{j_{1},\ldots,j_{k}}^{m_{1},\ldots,m_{k}}=0$.

Note that this implies that all summands that involve non-integer exponents of $\rme^{\rmi t}$ enter the sum with a factor equal to zero. Hence $g_{a,b}$ can be written as a polynomial in $u$, $\rme^{\rmi t}$ and $\rme^{-\rmi t}$.  

Consider a term $T_{j}^{m}$ which comes with a factor of $u^{s-1}$, so there is only one term which is not $u$ $(k=1)$, so $T_{j}^{m}=c_{m}\rme^{\frac{\rmi m(t+2\pi j)}{s}}$ for some $j=0,\ldots,s-1$. Then $T_{j}^{m}$ must appear with the factor $u^{s-1}$ for all values of $j$. We get 
\begin{equation}\sum_{j=0}^{s-1}T_{j}^{m}=c_{m}\rme^{\rmi mt/s}\sum_{j=0}^{s-1}\rme^{2\pi \rmi mj/s},\end{equation}
where the last sum is $0$ if $s\nmid m$ which shows the statement for $k=1$.

Assume now there is a $k\in\{1,2,\ldots,s-1\}$ such that all terms $T_{j_{1},\ldots,j_{k}}^{m_{1},\ldots,m_{k}}$ with $s\nmid\sum_{i=1}^{k}m_{i}$ cancel each other, i.e  
\begin{equation}
\label{eq:indhyp}\sum_{j_{1}=1}^{s}\underset{j_{2}\neq j_{1}}{\sum_{j_{2}=1}^{s}}\ldots\underset{j_{k}\neq j_{1},j_{2},\ldots,j_{k-1}}{\sum_{j_{k}=1}^{s}}T_{j_{1},\ldots,j_{k}}^{m_{1},\ldots,m_{k}}=0.\end{equation}

Let $T_{j_{1}\ldots,j_{k+1}}^{m_{1},\ldots,m_{k+1}}$ be a term with $s\nmid\sum_{i=1}^{k}m_{i}$. Summing over all terms with the same $(m_{1},\ldots,m_{k+1})$, but different choices of $(j_{1},\ldots,j_{k+1})$ yields 
\begin{equation}
\label{eq:sum}
\underset{j_{i} \text{ disjoint}}{\underset{\in(\mathbb{Z}/s\mathbb{Z})^{k+1}}{\sum_{(j_{1},j_{2},..,j_{k+1})}}}T_{j_{1},..,j_{k+1}}^{m_{1},..,m_{k+1}}=\rme^{\rmi t\sum_{p=1}^{k+1}\frac{m_{p}}{s}}\left(\prod_{p=1}^{k+1}c_{m_{p}}\right)\underset{j_{i}\text{ disjoint}}{\underset{\in(\mathbb{Z}/s\mathbb{Z})^{k+1}}{\sum_{(j_{1},j_{2},..,j_{k+1})}}}\prod_{p=1}^{k+1}\rme^{\frac{2\pi \rmi m_{p}j_{p}}{s}}.
\end{equation}
There is at least one $m_{p}$ which is not divisible by $s$. Otherwise $s|\sum_{p=1}^{k+1}m_{p}$. Without loss of generality, we can assume that $s\nmid m_{k+1}$. Then $\sum_{j=1}^{s}\rme^{\frac{2\pi \rmi m_{k+1}j}{s}}=0$ and hence $\rme^{\frac{2\pi \rmi m_{p}j_{p}}{s}}=-\sum_{p=1}^{k}\rme^{\frac{2\pi \rmi m_{k+1}j_{p}}{s}}-\underset{w\neq j_{1}\ldots,j_{k+1}}{\sum_{w=1}^{s}}\rme^{\frac{2\pi \rmi m_{k+1}w}{s}}$. Thus Eq. (\ref{eq:sum}) equals

\begin{equation}
\label{eq:t1t2}
\rme^{\rmi t\sum_{p=1}^{k+1}\frac{m_{p}}{s}}\left(\prod_{p=1}^{k+1}c_{m_{p}}\right)\underset{j_{i} \text{ disjoint}}{\underset{\in(\mathbb{Z}/s\mathbb{Z})^{k+1}}{\sum_{(j_{1},j_{2},..,j_{k+1})}}}
\left(\prod_{p=1}^{k}\rme^{\frac{2\pi \rmi m_{p}j_{p}}{s}}\times\left(-\sum_{p=1}^{k}\rme^{\frac{2\pi \rmi m_{k+1}j_{p}}{s}}-\underset{w\neq j_{1}\ldots,j_{k+1}}{\sum_{w=1}^{s}}\rme^{\frac{2\pi \rmi m_{k+1}w}{s}}\right)\right)=T_{1}+T_{2},
\end{equation}

with
\begin{equation}T_{1}=-\rme^{\rmi t\sum_{p=1}^{k+1}\frac{m_{p}}{s}}\left(\prod_{p=1}^{k+1}c_{m_{p}}\right)\underset{j_{i}\text{ disjoint}}{\underset{\in(\mathbb{Z}/s\mathbb{Z})^{k+1}}{\sum_{(j_{1},j_{2},..,j_{k+1})}}}\sum_{r=1}^{k}\rme^{2\pi \rmi(m_{r}+m_{k+1})j_{r}}\underset{p\neq r}{\prod_{p=1}^{k}}\rme^{\frac{2\pi \rmi m_{p}j_{p}}{s}}\end{equation}
and
\begin{equation}T_{2}= -\rme^{\rmi t\sum_{p=1}^{k+1}\frac{m_{p}}{s}}\left(\prod_{p=1}^{k+1}c_{m_{p}}\right)\underset{j_{i}\text{ disjoint}}{\underset{\in(\mathbb{Z}/s\mathbb{Z})^{k+1}}{\sum_{(j_{1},j_{2},..,j_{k+1})}}}\underset{w\neq j_{1}\ldots,j_{k+1}}{\sum_{w=1}^{s}}\rme^{\frac{2\pi \rmi m_{k+1}w}{s}}\prod_{p=1}^{k}\rme^{\frac{2\pi \rmi m_{p}j_{p}}{s}}.\end{equation}
For every fixed $r$ the first term $T_{1}$ describes 
$-\underset{j_{i}\text{ disjoint}}{\underset{\in(\mathbb{Z}/s\mathbb{Z})^{k}}{\sum_{(j_{1},j_{2},..,j_{k})}}}T_{j_{1},\ldots,j_{k}}^{m_{1},\ldots,m_{r}+m_{k+1},\ldots,m_{k}}$,
 which vanishes by the induction hypothesis (\ref{eq:indhyp}), since $s\nmid \sum_{p=1}^{k+1}m_{p}=(m_{r}+m_{k+1})+\underset{p\neq r}{\sum_{p=1}^{k}}m_{p}$ for all $r$.
 
For every fixed $w$ the second term $T_{2}$ is equal to 
$-\underset{j_{i}\text{ disjoint}}{\underset{\in(\mathbb{Z}/s\mathbb{Z})^{k+1}}{\sum_{(j_{1},j_{2},..,j_{k+1})}}}T_{j_{1},\ldots,j_{k+1}}^{m_{1},\ldots,m_{k+1}}$. To see this recall that we are summing over all possibilities to pick $k+1$ distinct values for $j_{p}$ between $1$ and $s$. The only difference between $T_{2}$ and Eq. (\ref{eq:sum}) is that we pick a number $j_{k+1}$ before we pick our last value $w$ and demand that this is distinct from $j_{k+1}$. Note that $w$ is playing the role of $j_{k+1}$ in Eq. (\ref{eq:sum}). Since we are also summing over all possible values for $j_{k+1}$, which are $s-k$ many, every possible $w$ enters the sum $s-k-1$ times.

Hence Eq. \ref{eq:t1t2} becomes
\begin{equation}\underset{j_{i}\text{disjoint}}{\underset{\in(\mathbb{Z}/s\mathbb{Z})^{k+1}}{\sum_{(j_{1},j_{2},..j_{k+1})}}}T_{j_{1},\ldots,j_{k+1}}^{m_{1},\ldots,m_{k+1}}=-(s-(k+1))\underset{j_{i}\text{disjoint}}{\underset{\in(\mathbb{Z}/s\mathbb{Z})^{k+1}}{\sum_{(j_{1},j_{2},..j_{k+1})}}}T_{j_{1},\ldots,j_{k+1}}^{m_{1},\ldots,m_{k+1}}.\end{equation} 

It is $k+1\leq s$ and hence 
\begin{equation}\sum_{j_{1}=1}^{s}\underset{j_{2}\neq j_{1}}{\sum_{j_{2}=1}^{s}}\ldots\underset{j_{k+1}\neq j_{1},\ldots,j_{k}}{\sum_{j_{k+1}=1}^{s}}T_{j_{1},\ldots,j_{k+1}}^{m_{1},\ldots,m_{k+1}}=0.\end{equation}
This completes the proof for $k+1$ and proves the lemma by induction.

Hence all terms of $g_{a,b}$ where $\rme^{\rmi t}$ has a non-integer exponent cancel each other when we expand the product. Therefore $g_{a,b}$ is a polynomial in $u$, $\rme^{\rmi t}$ and $\rme^{-\rmi t}$.
\end{proof}

As a polynomial in $u$, $\rme^{\rmi t}$ and $\rme^{-\rmi t}$ the function $g_{a,b}$ is the restriction of a polynomial $f_{a,b}:\mathbb{C}^{2}\rightarrow\mathbb{C}$ in $u$, $v$ and $\overline{v}$ to $\mathbb{C}\times S^{1}$. Thus $f_{a,b}(u,\rme^{\rmi t})=g_{a,b}(u,t)$ and by construction the zero level set of $f_{a,b}$ on $\mathbb{C}\times S^{1}$ is equal to the closure of the braid $g_{a,b}^{-1}(0)$ for every $a>0$, $b>0$.
The polynomial $f_{a,b}$ can be constructed from $g_{a,b}$ in the following way. We expand the product in Eq. (\ref{eq:braidpoly}), so that $g_{a,b}$ is in the form of a polynomial in $u$, $\rme^{\rmi t}$ and $\rme^{-\rmi t}$. Then we replace every instance of $\rme^{\rmi t}$ by $v$ and every instance of $\rme^{-\rmi t}$ by $\overline{v}$. Since $f_{a,b}(u,\rme^{\rmi t})=g_{a,b}(u,t)$ for all $u\in\mathbb{C}$ and $t\in[0,2\pi]$, the zero level set of $f_{a,b}$ on $\mathbb{C}\times S^{1}$ is the closure of the braid $B$ parametrised by Eq. (\ref{eq:fourierpara}).
Performing this substitution in Eq. (\ref{eq:braidpoly}) before the product is expanded results in a different function, which is typically not a polynomial in $u$, $v$ and $\overline{v}$.

Note that substituting $t$ by $rt$ in $F_{C}$ and $G_{C}$ for some $r\in\mathbb{N}$ results in a parametrisation of $r$ repeats of the original braid. Thus if $f_{a,b}$ has a zero level set on $\mathbb{C}\times S^{1}$ which is of the form of the closure of a braid $B$, then multiplying each exponent of $v$ and $\overline{v}$ by $r$ gives a semiholomorphic polynomial with zero level set on $\mathbb{C}\times S^1$ which is the closure of $B^r$. 

The classes of knots that were discussed in \cite{us:2016lemniscate} can now be seen as particularly simple examples, since lemniscate knots are exactly the closures of repeats of braids given by parametrisations of the form of Eq. (\ref{eq:fourierpara}) with $F(t)=\cos(t)$ and $G(t)=\sin(\ell t)$ for some $\ell\in\mathbb{N}$ coprime to the number of strands $s$, generalising the torus knots with $\ell=1$. The spiral knots, first introduced in \cite{betvwy:2010spiral}, are exactly the closures of repeats of braids that can be parametrised by Eq. (\ref{eq:fourierpara}) with $F(t)=\cos(t)$ and $G(t)$ some trigonometric polynomial.

By Alexander's Theorem \cite{alexander:1923lemma} every link is the closure of a braid and since trigonometric polynomials are dense in the set of continuous $2\pi$-periodic functions from $\mathbb{R}$ to $\mathbb{R}$, every braid can be parametrised as in Eq. (\ref{eq:fourierpara}). Hence for every link $L$ there exists a family of semiholomorphic polynomials $f_{a,b}$ with $f_{a,b}^{-1}(0)\cap(\mathbb{C}\times S^{1})=L$ for all $a>0$, $b>0$.

In Section \ref{proof} we show that for small enough $a$ and $b$, it is in fact $f_{a,b}^{-1}(0)\cap S^{3}=L$ as desired.

\subsection{The proof of Theorem \ref{thmintro}}
\label{proof}

Let $f_{a,b}:\mathbb{C}^{2}\rightarrow\mathbb{C}$ be the polynomial constructed from $g_{a,b}$, so $f_{a,b}^{-1}(0)\cap (\mathbb{C}\times S^{1})$ is the closure of the braid that is $g_{a,b}^{-1}(0)$. We will make this statement more precise. Consider the map $\Psi:\mathring{D}\times(\mathbb{C}\backslash\{0\})\to S^{3}$,  $\Psi(u,r\rme^{\rmi t})=(u,\sqrt{1-|u|^{2}}\rme^{\rmi t})$, where $\mathring{D}$ is the interior of the unit disk in the complex plane. 
We think of $f_{a,b}$ as a family of polynomials, parametrised by $v=r\rme^{\rmi t}$, $a$ and $b$, in one complex variable $u$. In the following we set $a=\lambda a_{1}$ and $b=\lambda b_{1}$, leave $a_{1}$ and $b_{1}$ as real constants and only vary the real parameter $\lambda$. With this notation we write $f_{\lambda}$ instead of $f_{a,b}$, slightly abusing notation.
For fixed $r$, $t$ and $\lambda$, we denote the $s$ roots of $f_{\lambda}(\bullet,r\rme^{\rmi t})$ by $u_{\lambda,j}(r,t)$. Note that by definition $u_{\lambda,j,}(r,t)=\lambda u_{1,j}(r,t)$., $j=1,2,\ldots,s$. Thus if $\lambda$ is small enough, all roots $u_{\lambda,j}(1,\rme^{\rmi t})$ of $f_{\lambda}(\bullet,\rme^{\rmi t})$ lie in $\mathring{D}$. It is intuitively clear and can be shown that $\Psi(f_{\lambda}^{-1}(0)\cap(\mathbb{C}\times S^{1}))=\Psi((u_{\lambda,j}(1,\rme^{\rmi t}),\rme^{\rmi t}))$ is the closure of $B=g_{a,b}^{-1}(0)$, which is $L$. This idea can be found among others in \cite{perron:1982noeud}.

In order to show that $f_{\lambda}^{-1}(0)\cap S^{3}$ is $L$ as well, we need to show that there is an ambient isotopy between $L_{1}=\Psi(f_{\lambda}^{-1}(0)\cap(\mathbb{C}\times S^{1}))$ and $L_{2}=f_{\lambda}^{-1}(0)\cap S^{3}$.

\begin{defi}Two links $L_{1}$ and $L_{2}$ are ambient isotopic if there is a continuous function $F:S^{3}\times[0,1]\to S^{3}$, such that $F(\cdot,0)$ is the identity map, $F(L_{1},1)=L_{2}$ and $F(\cdot,s)$ is a homeomorphism for all $s\in[0,1]$.
\end{defi}

In fact by the Isotopy Extension Theorem, it is enough to construct a smooth isotopy between the two sets of curves.

\begin{thm}
\label{ext}
\cite{ek:1971imbeddings}
Let $I:S^{1}\cup\ldots\cup S^{1}\times[0,1]\to S^{3}$ be a smooth 1-parameter family of embeddings of $n$ circles in $S^{3}$, such that $I(S^{1},\ldots,S^{1},0)=L_{1}$ and $I(S^{1},\ldots,S^{1},1)=L_{2}$, then there exists an ambient isotopy $\tilde{I}:S^{3}\times[0,1]\to S^{3}$ with $\tilde{I}(L_{1},s)=I(S^{1},\ldots,S^{1},s)$.
\end{thm}

For fixed $\lambda$, $j$ and $h$ we think of the roots $u_{\lambda,j}(r,t)$ as functions of $r$. 
Note that the union of the intersection points of the roots $u_{\lambda,j}(r,t)$ with $S^3$ is equal to the zero level set of $f_{\lambda}$ on $S^{3}$.
We would like to see that for every fixed $\lambda$, $j$ and $t$ there is a unique intersection point of $(u_{\lambda,j}(r,t)r\rme^{\rmi h})$ with $S^{3}$, so that there is a $1-1$ correspondence between the points in $f_{\lambda}^{-1}(0)\cap S^{3}$ and $f_{\lambda}^{-1}(0)\cap(\mathbb{C}\times S^{1})$.
To do this we first need to restrict the range of $r$ to a domain where all the roots are disjoint. This allows us to treat the roots $u_{\lambda,j}(r,t)$ as smooth functions of $r$.

\begin{lemma}
\label{disjoint} There is a $\delta>0$ independent of $\lambda$ such that if $u_{\lambda,j}(r_{1},t)=u_{\lambda,k}(r_{2},t)$ with $r_{1},r_{2}\in[1-\delta,1]$, then $j=k$.
\end{lemma}

\begin{proof}
Note that it is enough to show the lemma for $\lambda=1$, since $u_{\lambda,j}(r,t)=\lambda u_{1,j}(r,t)$. In particular the value that we find for $\delta$ when $\lambda=1$ will be sufficient for any choice of $\lambda>0$.

Let $R$ be the biggest value of $r$ for which different strands intersect for $\lambda=1$, that is $R=\max_{t\in[0,2\pi],\ j\neq k}\{r|\ u_{1,j}(r,t)=u_{1,k}(r,t),\ r\leq1\}$. Since $u_{\lambda,j}(r,t)=\lambda u_{1,j}(r,t)$ for all $\lambda$, $j$, $r$ and $t$, all the roots $u_{\lambda,j}(r,t)$ are simple roots of $f_{\lambda}(\bullet,r\rme^{\rmi t})$ as long as $r\in[R,1]$ and hence $\re(u_{\lambda,j}(r,t))$ and $\im(u_{\lambda,j}(r,t))$ are differentiable functions with respect to $r\in[R,1]$.

The Implicit Function Theorem allows us to calculate these derivatives in terms of $\frac{\partial f_{a}(\cdot,r\rme^{\rmi t})}{\partial u}$ and $\frac{\partial f_{a}(\cdot,r\rme^{\rmi t})}{\partial r}$. Let $D_{1}(\delta)=\max_{t\in[0,2\pi],\ j,\ r\in[1-\delta,1]}\left\{\left|\frac{\partial \re(u_{1,j}(r,t))}{\partial r}\right|\right\}$ and $D_{2}(\delta)=\max_{t\in[0,2\pi],\ j,\ r\in[1-\delta,1]}\left\{\left|\frac{\partial \im(u_{1,j}(r,t))}{\partial r}\right|\right\}$ for $\delta\in[0,1-R]$.

Take $\tilde{\delta}$ to be $\frac{1}{2}\underset{j\neq k}{\min_{t\in[0,2\pi]}}\{|u_{1,j}(1,t)-u_{1,k}(1,t)|\}$ and choose $\delta<\frac{\tilde{\delta}}{\sqrt{D_{1}(\delta)^2+D_{2}(\delta)^2}}$. Note that as $\delta$ approaches zero, the right hand side of the inequality converges to a non-zero value. Hence such a $\delta$ always exists. The Lemma follows then from 
\begin{align}
|u_{1,j}(r_{1},t)-u_{1,k}(r_{2},t)|&\geq|u_{1,j}(1,t)-u_{1,k}(1,t)|-|u_{1,j}(1,t)-u_{1,j}(r_{1},t)|\nonumber\\
&\ \ \ -|u_{1,k}(1,t)-u_{1,k}(r_{2},t)|\nonumber\\
&\geq|u_{1,j}(1,t)-u_{1,k}(1,t)| -2\delta |D_{1}(\delta)+\rmi D_{2}(\delta)|\nonumber\\
&>|u_{1,j}(1,t)-u_{1,k}(1,t)|-2\tilde{\delta}\geq0
\end{align} for all $r\in[1-\delta,1]$ and $j\neq k$.
\end{proof}

We use the following lemma to show that for every fixed $\lambda$, $j$ and $h$ the intersection point of $(u_{\lambda,j}(r,t),r\rme^{\rmi t})$ with $S^{3}$ is unique if $\lambda$ is small enough.

\begin{lemma}
\label{atmost} There exists $\varepsilon_{1}>0$ such that for every fixed $\lambda<\varepsilon_{1}$ and all fixed $j$ and $t$ there is at most one intersection of $(u_{\lambda,j}(r,t),r\rme^{\rmi t})$ with $S^{3}$ where $1-\delta\leq r\le1$.
\end{lemma}

\begin{proof}Lemma \ref{disjoint} implies that the roots $u_{\lambda,j}(r,t)$, $j=1,2,\ldots,s$ are disjoint for $r\in[1-\delta,1]$. As simple roots of the polynomial $f_{\lambda}(\bullet,r\rme^{\rmi t})$, they depend smoothly on its coefficients, in particular $u_{\lambda,j}(r,t)$ is a smooth function of $r$ for all $r\in[1-\delta,1]$.

Consider the function $|u_{\lambda,j}(r,t)|^{2}+r^{2}-1=\lambda^{2}|u_{1,j}(r,t)|^{2}+r^{2}-1$. This is a smooth function of $r$ on $[1-\delta,1]$ and its zeros correspond to intersection points of $u_{\lambda,j}(r,t)$ with $S^{3}$.

Now suppose a curve $(u_{\lambda,j}(r,t),r\rme^{\rmi t})$ with fixed $\lambda$, $j$ and $h$ has multiple intersection points with $S^{3}$ while $r\in[1-\delta,1]$. Then between these the function $|u_{\lambda,j}(r,t)|^{2}+r^{2}-1$ must have an extremum.

We can choose $\lambda$ small enough such that the derivative $\frac{d}{dr}(\lambda^{2}|u_{1,j}(r,t)|^{2}+r^{2}-1)$ is strictly positive for all $r\in[1-\delta,1]$. Hence for small $\lambda$, there is a unique intersection point with $r$ in that interval.

In fact we can find a sufficient value for $\varepsilon_{1}$ as follows.
 
We need to make sure that $\frac{d}{dr}(\lambda^{2}|u_{1,j}(r,t)|^{2}+r^{2}-1)>0$ for all $r\in[1-\delta,1]$ and all $\lambda<\varepsilon_{2}$. Define $\tilde{D}_{1}=\underset{r\in[1-\delta,1]}{\underset{j=0,1,\ldots,s-1}{\max_{t\in[0,2\pi]}}}\left\{\left|\frac{\partial |\re(u_{1,j}(r,t))|}{\partial r}\right|\right\}$ and $\tilde{D}_{2}=\underset{r\in[1-\delta,1]}{\underset{j=0,1,\ldots,s-1}{\max_{t\in[0,2\pi]}}}\left\{\left|\frac{\partial |\im(u_{1,j}(r,t))|}{\partial r}\right|\right\}$. Note that in general $D_{i}\neq\tilde{D}_{i}$. With $\varepsilon_{1}=\sqrt{\frac{(1-\delta)}{U(\tilde{D}_{1}+\tilde{D}_{2})}}$ we find 
\begin{equation}\frac{d}{dr}(\lambda^{2}|u_{1,j}(r,t)|^{2}+r^{2}-1)>2r-\lambda^{2}2U(\tilde{D}_{1}+\tilde{D}_{2})>0
,\end{equation} 
as long as $\lambda<\varepsilon_{1}$ and $r\in[1-\delta,1]$.

\end{proof}

Recall that $\delta$ from Lemma \ref{disjoint} does not depend on $\lambda$. This allows us to prove the following lemma.
\begin{lemma}
\label{intersect}
There exists $\varepsilon_{2}>0$ such that for all $\lambda<\varepsilon_{2}$ all intersections of the curves $u_{\lambda,j}(r,t)$ with $S^{3}$ occur with $r\in[1-\delta,1]$.
 \end{lemma}
\begin{proof}
Since $f_{\lambda}(\bullet,r\rme^{\rmi t})$ is a polynomial in the first variable $u$, we may write $f_{\lambda}(u,r\rme^{\rmi t})=\sum_{j=0}^{s}c_{j}\lambda^{s-j}u^{j}$, where $c_{j}\in\mathbb{C}$ depends on $r$ and $t$. Note that it follows from Rouch\'{e}s Theorem that 
\begin{equation}U=\underset{t\in[0,2\pi]}{\underset{j\in\{0,1,\ldots,s-1\}}{\max_{r\in[1-\delta,1]}}}\{1,\sum_{j=0}^{s-1}|c_{j}|\}\end{equation}
is a bound on the modulus of all roots of all polynomials $f_{1}(\bullet,r\rme^{\rmi t})$ with $r\in[1-\delta,1]$. Hence $\lambda U$ is an upper bound on the modulus of all roots of all these polynomials $f_{\lambda}(\bullet,r\rme^{\rmi t})$, which are by definition $u_{\lambda,j}(r,t)$.

We can now make sure that all  intersection points of the roots $u_{\lambda,j}(r,t)$ with $S^{3}$ happen in the region where $r\in[1-\delta,1]$ by setting $\varepsilon_{2}=\frac{\delta}{U}$. Then for all $\lambda<\varepsilon_{2}$, we have 
\begin{equation}
\label{eq:inside}
|u_{\lambda,j}(r,t)|^{2}+r^{2}<\lambda^{2}U^{2}+1-\delta<1
\end{equation}
for all $r<1-\delta$ and hence all intersections with $S^{3}$ happen for $r\in[1-\delta,1]$.

\end{proof}

Note that Eq. (\ref{eq:inside}) also shows that for every $\lambda<\varepsilon_{2}$ and every fixed $j$ and $h$, the curve $u_{\lambda,j}(r,t)$ intersects $S^3$, since $|u_{\lambda,j}(r,t)|^2+r^2\geq1$ at $r=1$.
Combining this with Lemma \ref{atmost} means that this intersection point is unique for every curve $(u_{\lambda,j}(r,t),r\rme^{\rmi t})$. Denote the value of $r$ at which the intersection occurs by $r_{\lambda,j}(t)$.

Recall the definition of the map $\Psi:\mathring{D}\times(\mathbb{C}\backslash\{0\})\to S^{3}$, $\Psi(u,r\rme^{it})=(u,\sqrt{1-|u|^{2}}\rme^{\rmi t})$. We now define a smooth isotopy between $\Psi(f_{\lambda}^{-1}(0)\cap(\mathbb{C}\times S^{1}))$, which we know to be the desired link, and $f_{\lambda}^{-1}(0)\cap S^{3}$.

Let $I:(\underbrace{S^{1}\cup\ldots\cup S^{1}}_{|\mathcal{C}|\text{ copies}})\times[0,1]$ be defined by 
\begin{equation}
I(\Psi((u_{\lambda,j}(1,t),\rme^{\rmi t})),s)=\Psi(u_{\lambda,j}(1-s+sr_{\lambda,j}(t),t),(1-s+sr_{\lambda,j}(t))\rme^{\rmi t}).
\label{eq:defpsi}
\end{equation}

\begin{lemma}
\label{isotopy} There is a $\varepsilon>0$ such that for all $\lambda<\varepsilon$ the function $I$ is a smooth isotopy from $\Psi(f_{\lambda}^{-1}(0)\cap(\mathbb{C}\times S^{1}))$ to $f_{\lambda}^{-1}(0)\cap S^{3}$.
\end{lemma}

\begin{proof}We have to show that for every $s$ the function $I(\bullet,s)$ is an embedding and that it is smooth with respect to $s$. Suppose there is an $s$ such that $I(\cdot,s)$ is not an embedding, then there exist complex numbers $u_{\lambda,j}(1-s+sr_{\lambda,j}(t),t)=u_{\lambda,k}(1-s+sr_{\lambda,k}(t),t)$ with $j\neq k$, but Lemma \ref{disjoint} and Lemma \ref{intersect} tell us that this does not happen if $\lambda<\min\{\varepsilon_{1},\varepsilon_{2}\}$. Furthermore, since simple roots of a polynomial depend smoothly on its coefficients and the map $\Psi$ is smooth, $t$ is a smooth isotopy.

Note that 
\begin{equation}I(\Psi((u_{\lambda,j}(1,t),\rme^{\rmi t})),0)=\Psi(u_{\lambda,j}(1,t),\rme^{\rmi t})=\Psi(f_{\lambda}^{-1}(0)\cap (\mathbb{C}\times S^{1}))\end{equation} 
and
\begin{equation}I(\Psi((u_{\lambda,j}(1,t),\rme^{\rmi t})),1)=\Psi(u_{\lambda,j}(r_{\lambda,j}(t),t),r_{\lambda,j}(t)\rme^{\rmi t})=f_{\lambda}^{-1}(0)\cap S^{3}\end{equation} 
which finishes the proof.
Note that $I$ is well-defined if $\lambda<\min\{\varepsilon_{1},\varepsilon_{2}\}$, so we can set $\varepsilon=\min\{\varepsilon_{1},\varepsilon_{2}\}$. 
\end{proof}

The Isotopy Extension Theorem says that $I$ extends to an ambient isotopy. Thus $f_{\lambda}^{-1}(0)\cap S^{3}$ is ambient isotopic to the desired link if $\lambda$ is small enough.

Since every link is the closure of a braid by Alexander's Theorem \cite{alexander:1923lemma} and for every braid $B$ there is a family of functions $f_{\lambda}$, this concludes the proof of Theorem \ref{thmintro}.
We have shown that for every link $L$ there is a function $f:\mathbb{C}^{2}\rightarrow\mathbb{C}$ such that $f^{-1}(0)\cap S^{3}=L$ and $f$ is a polynomial in complex variables $u$, $v$ and $\overline{v}$, i.e. $f$ is a semiholomorphic polynomial.

In \cite{us:2016fivetwo} we describe the constructed semiholomorphic polynomials for many knots and links of low minimal crossing number.
It turns out that in practice $\lambda$ can be chosen to be a lot larger than the bound $\varepsilon=\min\{\varepsilon_{1},\varepsilon_{2}\}$ which is given in the proof.

\section{Finding Fourier parametrisations}
\label{algo}

In this section we present an algorithm that generates a parametrisation as in Eq. (\ref{eq:fourierpara}) for any given braid.

For every link component $C\in\mathcal{C}$ we obtain $F_{C}$ and $G_{C}$ by trigonometric interpolating data points that can be obtained from a presentation of the braid diagram. For background on trigonometric interpolation we point the reader to \cite{atkinson:1988numerical}.
\ \\

\noindent\textbf{Step 1: Finding the data points for the trigonometric interpolation for $F_{C}$}

We need to perform a trigonometric interpolation for $F_{C}$ for each link component $C\in\mathcal{C}$. The data points for this interpolation are chosen such that they contain the information about the position of every strand of $C$ between two crossings in the braid diagram.

Let $s$ denote the number of strands and $\ell$ the length of the braid word. The given braid word allows us to draw a braid diagram of $B$. For convenience we draw the strands as piecewise linear curves with all crossing points evenly distributed along the $t$-axis as in Fig. \ref{fig:interpolation} a). If we neglect the signs of the crossings, the diagram consists of $s$ curves $(D_{C,j}(t),0,t)$, $t\in[0,2\pi]$, $C\in\mathcal{C}$, $j=0,1,\ldots,s_{C}$ in the $tx$-plane, each of which can be interpreted as the graph of a function $D_{C,j}(t)$ (Fig. \ref{fig:interpolation} b)). Since the crossing points are evenly distributed, they occur at $t_{k}=2\pi(2k-1)/(2\ell)$ with $k=1,2,\ldots \ell$. The braid diagram is drawn in such a way that the value of $D_{C,j}(2\pi)$ is equal to $D_{C,k}(0)$ for some $k$, say $k=j+1$. This way each strand obtains a label $(C,j)$. This label is unique once for each component $C$ an arbitrary strand of that component is chosen as the strand $(C,1)$. We can draw the graph of the piecewise-linear function $D_{C}:[0,2\pi]\to\mathbb{R}$, $D_{C}(\frac{t+2\pi j}{s_{C}})=D_{C,j}(t)$ as in Fig. \ref{fig:interpolation} c).

We can now perform a trigonometric interpolation through the points $(t_{k}/s_{C}-2\pi/(2s_{C}\ell), D_{C}(t_{k}-2\pi/(2\ell)))$, $k=1,2,\ldots,s_{C}\ell$ (shown in Fig. \ref{fig:interpolation} d)) for every component $C$ to obtain $|\mathcal{C}|$ trigonometric polynomials $F_{C}$ that satisfy $F_{C}(t_{k}/s_{C}-2\pi/(2s_{C}\ell))=D_{C}(t_{k}-2\pi/(2\ell))$ for every $k=1,2,\ldots s_{C}\ell$ and every link component $C\in\mathcal{C}$ (cf. Fig. \ref{fig:interpolation} e)).
\ \\

\noindent\textbf{Step 2: Trigonometric interpolation for $F_{C}$} 

The trigonometric polynomials $F_{C}$ can be explicitly constructed by using the discrete Fourier transform 
\begin{equation}\tilde{D}_{C,k}=\frac{1}{s_{C}\ell}\sum_{n=0}^{s_{C}\ell-1}D_{C}\left(t_{n}-\frac{2\pi}{2s_{C}\ell}\right)\rme^{-\rmi\left(\frac{t_{n}}{s_{C}}-\frac{2\pi}{2s_{C}\ell}\right)k}.\end{equation}

We get 
\begin{equation}F_{C}(t)=\begin{cases}\tilde{D}_{C,0}+\sum_{k=1}^{s_{C}\ell/2-1}2\re(\tilde{D}_{C,k})\cos(kt)-\sum_{k=1}^{s_{C}\ell/2-1}2\im(\tilde{D}_{C,k})\sin(kt)\\
\qquad\ +\tilde{D}_{C,s_{C}\ell/2}\cos\left(\frac{s_{C}\ell}{2}t\right),\qquad\text{if } s_{C}\ell\text{ is even},\\
\ \\
\tilde{D}_{C,0}+\sum_{k=1}^{(s_{C}\ell-1)/2}2\re(\tilde{D}_{C,k})\cos(kt)-\sum_{k=1}^{(s_{C}\ell-1)/2}2\im(\tilde{D}_{C,k})\sin(kt),\\ 
\hspace{4.9cm}\text{if } s_{C}\ell\text{ is odd.}\end{cases}\end{equation}

Since the example in Fig. \ref{fig:interpolation} is the knot $5_{2}$, there is only one link component. 
Figure \ref{fig:interpolation} e) shows the graph of the trigonometric polynomial $F_{C}$. The graphs of $F_{C}\left(\frac{t+2\pi j}{s_{C}}\right)$ in Fig. \ref{fig:interpolation} f) form a braid diagram $B'$ with unspecified signs of crossing. Note that in contrast to usual braid diagrams there might in general be more than two strands involved in a crossing and crossing strands might not be transverse.

Since the trigonometric polynomials $F_{C}$ are interpolating the data points, there is a bijection between the strands of $B'$ and the strands of the original braid $B$ and for every $k$ there is a crossing of $B'$ in the interval $\left[t_{k}-2\pi/(2\ell),t_{k+1}-2\pi/(2\ell)\right]$ that involves the same strands as the crossing of $B$ at $t_{k}$. This crossing might not be unique. However, for every pair of strands of $B$ that is not crossing at $t_{k}$, there is an even number of crossings in the diagram $B'$ between them in the interval $\left[t_{k}-2\pi/(2\ell),t_{k+1}-2\pi/(2\ell)\right]$ (counting multiplicities if a crossing is not transverse). For the pair of strands of $B$ that is crossing at $t_{k}$, the corresponding strands of $B'$ will cross an odd number of times in the interval $\left[t_{k}-2\pi/(2\ell),t_{k+1}-2\pi/(2\ell)\right]$ (counting multiplicities). This is due to the choice of data points for the trigonometric interpolation that store the information about the position of every strand of $B$ between two crossings.

\ \\

\noindent\textbf{Step 3: Finding the data points for the trigonometric interpolation for $G_{C}$}

Again we need to perform a trigonometric interpolation for each link component $C\in\mathcal{C}$. Recall that Step 2 resulted in $B'$, a braid diagram whose signs of crossings are not specified.
By choosing data points whose $t$-coordinates are the positions of crossings in the braid diagram $B'$ and choosing the $y$-coordinate appropriately, we attach signs to the crossings of $B'$ in such a way that it becomes the diagram of a braid isotopic to $B$.

This can be achieved as follows. For every $k=1,2,\ldots \ell$ we choose a bijection $w_{k}$ between the strands of $B'$ and a set of $s$ distinct real numbers such that the strand corresponding to the strand of $B$ which is overcrossing at $t_{k}$ gets assigned a larger number than the strand corresponding to the strand of $B$ which is undercrossing at $t_{k}$.

This time the $t$-coordinate of the data points used for the trigonometric interpolation for $G_{C}$ are the positions of the crossings that the strands $F_{C}\left(\frac{t}{s_{C}}+\frac{2\pi j}{s_{C}}\right)$ are involved in. Let the strand $(C,j)$ be involved in a crossing of $B'$ at $t=t_{0}$, i.e. $F_{C}\left(\frac{t_{0}}{s_{C}}+\frac{2\pi j}{s_{C}}\right)=F_{C'}\left(\frac{t_{0}}{s_{C'}}+\frac{2\pi m}{s_{C'}}\right)$ for some $C'\in\mathcal{C}$ and $m\in\{0,1,\ldots,s_{C'}-1\}$. Then $\left(\frac{t_{0}+2\pi j}{s_{C}},w_{k}((C,j))\right)$ is a data point. Note that if a strand is not involved in a crossing in the $k$th interval, then there is no data point for this strand in this interval and hence its image under $w_{k}$ does not affect the result at all.
We denote the data points by $(t_{k}',y_{k})$. The number of data points for $G_{C}$ depends on the number of crossings of $B'$ that strands of the link component $C$ are involved in.

Each crossing creates a data point for each strand involved in the crossing. This choice of data points implies that any family of trigonometric polynomials $G_{C}$ that interpolates these together with $F_{C}$ from Step 1 gives a parametrisation of a braid as in Eq. (\ref{eq:fourierpara}) and this braid is braid isotopic to $B$.
This can be seen as follows. Since the $y$-coordinate of the data points only depends on the strand and on $k$, all crossings between one pair of strands in the interval $\left[t_{k}-2\pi/(2\ell),t_{k+1}-2\pi/(2\ell)\right]$ in $B'$ have the same strand as an overcrossing strand (and the other as the undercrossing strand). This and the previous observation about the parity of numbers of crossings between strands means that in every interval $\left[t_{k}-2\pi/(2\ell),t_{k+1}-2\pi/(2\ell)\right]$ all crossings but one can be canceled. The two strands that are involved in the remaining crossing correspond to the two strands of $B$ crossing at $t_{k}$ and by definition of $w_{k}$ they cross with the required sign.
\ \\

\noindent\textbf{Step 4: Trigonometric interpolation for $G_{C}$}

Since the positions of the crossings of $B'$ are in general not equidistributed, the trigonometric interpolation does not directly translate to a discrete Fourier Transform.

Finding a trigonometric polynomial that interpolates the data points $(t_{k}',y_{k})$ is equivalent to finding a function $q(z)=\sum_{k=-M_{C}}^{M_{C}}c_{k}z^{k}$ with $q(\rme^{\rmi t_{k}'})=y_{k}$ for all $k$, which is equivalent to finding a polynomial $\tilde{q}\sum_{k=0}^{2M_{C}}c_{k-K}z^{k}$ such that $q(\rme^{\rmi t_{k}'})=y_{k}\rme^{\rmi Kt_{k}'}$ for all $k$ and $c_{-k}=\overline{c}_{k}$ for all $k$.

We can find such a function by using the Lagrange formula for polynomial interpolation. This allows us to explicitly calculate $G_{C}$. If the number $N$ of data points $(t_{k}',y_{k})$ is odd, say $2K=1$, we get 
\begin{equation}q(\rme^{\rmi t})=G_{C}(t)=\sum_{k=1}^{N}y_{k}\rme^{-\rmi Kt+\rmi Kt_{k}'}\underset{m\neq k}{\prod_{m=1}^{N}}\frac{\rme^{\rmi t}-\rme^{\rmi t_{m}'}}{\rme^{\rmi t_{k}'}-\rme^{\rmi t_{m}'}}.\end{equation} 
If $N$ is even, say $N=2K$, the result is 
\begin{equation}q(\rme^{\rmi t})=G_{C}(t)=\sum_{k=1}^{N}y_{k}\rme^{-\rmi Kt+\rmi Kt_{k}'}\frac{\rme^{\rmi t}-\rme^{\rmi\alpha_{k}}}{\rme^{\rmi t_{k}'}-\rme^{\rmi\alpha_{k}}}\underset{m\neq k}{\prod_{m=1}^{N}}\frac{\rme^{\rmi t}-\rme^{\rmi t_{m}'}}{\rme^{\rmi t_{k}'}-\rme^{\rmi t_{m}'}},\end{equation} 
where $\alpha_{k}$ is a free parameter that can be chosen to be $0$ for all $k$.

Writing $q(\rme^{\rmi t})=\sum_{k=-K}^{K}c_{k}\rme^{\rmi kt}$ as a power series in $\rme^{\rmi t}$ allows to compute the coefficients of $G_{C}$ for every link component $C\in\mathcal{C}$.

Then Eq. (\ref{eq:fourierpara}) with the trigonometric polynomials $F_{C}$ and $G_{C}$ obtained by trigonometric interpolation is a finite Fourier parametrisation of the desired braid.

\ \\
In this construction there are several choices to be made, each of which leads to different trigonometric polynomials. There is first of all the choice of braid. We will see in Section \ref{Props} how the choice of braid affects the possible degrees of the resulting trigonometric polynomials. It is not true in general that the simplest braid (in terms of number of crossings $\ell$ or number of strands $s$) leads to the simplest trigonometric polynomials (in terms of degree). However, we compute bounds for the degree of the trigonometric polynomials in terms of $\ell$ and $s$ and these are strictly increasing with $\ell$ and $s$.

Secondly, there is the choice of embedding the linear braid diagram in $\mathbb{R}^{2}$, which corresponds to the functions $D_{C,j}$. Demanding that between crossings the strands are equidistant and are arranged symmetrically around $x=0$ seem like reasonable conditions. The choice of function $D_{C}$ also means that we choose a first strand for each link component. The algorithm works for any such choice.

Furthermore, the choice of functions $w_{k}$ that are used to determine the data points for the interpolation for $G_{C}$ do not affect the topology of the braid parametrised by the obtained trigonometric polynomials. Like in the choice of the embedding it appears reasonable to place the strands that are involved in crossings in the relevant interval symmetrically and with equal distances around $y=0$.

\begin{figure}
\centering
\subfloat{\includegraphics[height=2.5cm]{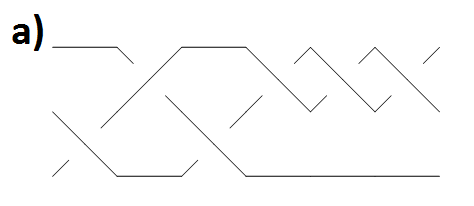}}\quad
\subfloat{\label{fig:b}\includegraphics[height=3cm]{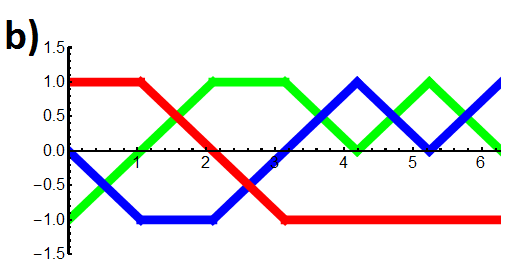}}\\
\hspace{-0.5cm}
\subfloat{\label{fig:c}\includegraphics[height=3cm]{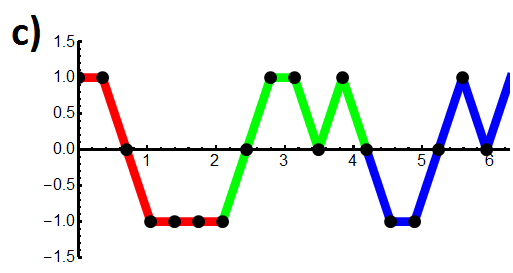}}\hspace{0.18cm}\quad
\subfloat{\label{fig:d}\includegraphics[height=3cm]{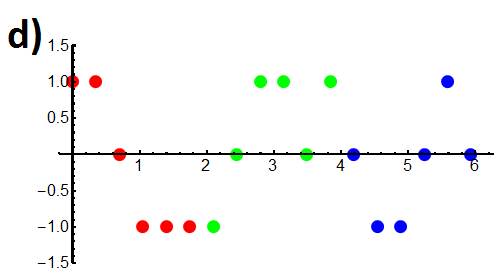}}\\
\hspace{-0.5cm}
\subfloat{\label{fig:e}\includegraphics[height=3.1cm]{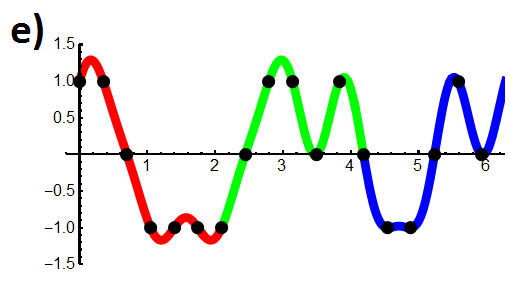}}\hspace{0.14cm}\quad
\subfloat{\label{fig:f}\includegraphics[height=3cm]{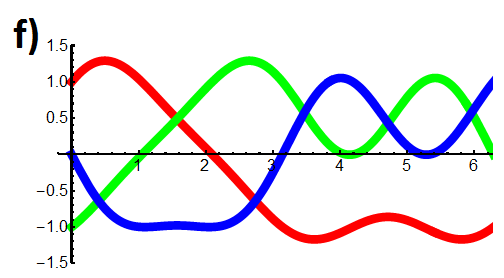}}

\caption{The interpolation for $F$. a) The braid diagram of a minimal braid of $5_{2}$. b) The projection of strands of the braid as piecewise linear curves in a coordinate system. c) The graph of the function $D_{C}$. d) The data points for the trigonometric interpolation. These can be read off from the graph of $D_{C}$. e) The graph of $F$, the trigonometric polynomial interpolating the data points. f) The graphs of the different $X_{j}$, when $F$ is used to parametrise the $x$-coordinate of the braid. The crossing points in this plot give rise to the data points for the interpolation for $G$.}
\label{fig:interpolation}
\end{figure}
\ \\

In \cite{us:2016lemniscate} we described two families of knots, lemniscate knots and spiral knots, whose particular symmetric braid words allowed comparatively simple trigonometric braid parametrisations.

Figure \ref{fig:interpolation} illustrates our method for the example of the knot $5_{2}$, which is the simplest knot (in terms of minimal crossing number) that is not a lemniscate knot or a spiral knot. For this example we choose the braid $\sigma_{2}\sigma_{1}^{-1}\sigma_{2}\sigma_{1}^{3}$ (Fig. \ref{fig:interpolation} a)) which closes to the knot $5_{2}$. In fact it is a minimum braid word of $5_{2}$ \cite{gittings:2004minimum}.
We read off braid words from braid diagram from left to right and strand positions are numerated from top to bottom. A positive crossing $\sigma_{i}$ corresponds to the strand in position $i$ passing over the strand in position $i+1$ from the top.

Since $5_{2}$ is a knot, there is only one component. To find $F_{C}$, we interpolate the data points $(\tilde{t}_{k},x_{k})$ for $k=0,1,\ldots,17$, where $\tilde{t}_{k}=2\pi k/18$ and $x_{k}$ is the $k$th entry of $k=(1,1,0,-1,-1,-1,-1,0,1,1,0,1,0,-1,-1,0,1,0)$. These data points can be read off from Fig. \ref{fig:interpolation} c) and are displayed in Fig. \ref{fig:interpolation} d).

The resulting interpolating trigonometric polynomial is 

\begin{align}
F(t)&=0.108784 \rme^{-8\rmi t}+0.0846189 \rmi \rme^{-7\rmi t}+0.24365\rmi \rme^{-5\rmi t}-0.0886815\rme^{-4\rmi t}\nonumber\\
&\ \ \ +0.479898\rme^{-2\rmi t}-0.129644\rmi \rme^{-\rmi t}+0.129644\rmi \rme^{\rmi t}+0.479898\rme^{2\rmi t}\nonumber\\
&\ \ \ -0.0886815\rme^{4\rmi t}-0.24365\rmi  \rme^{5\rmi t}-0.0846189 \rmi \rme^{7\rmi t}+0.108784 \rme^{8\rmi t}\nonumber\\
&=-0.259288\sin(t)+0.959796\cos(2t)-0.177363\cos(4t)+0.4873\sin(5t)\nonumber\\
&\ \ \ +0.169238\sin(7t)+0.217568\cos(8t).
\end{align}

The coefficients were computed numerically, so they are not necessarily exact. The graph of $F$ is plotted in Fig. \ref{fig:interpolation} e). In order to find the data points for the trigonometric interpolation of $G$, we need to compute the values of $h$ for which $F(t+2\pi j/3)=F(t+2\pi k/3)$ for some $j,k\in\{0,1,2\}$.

There is only one crossing between the red strand $(j=0)$ and the green strand $(j=1)$, namely at $t=0.523599$. Since this crossing corresponds to the crossing $\sigma_{1}^{-1}$ in the original braid word, the green strand is supposed to pass over the red strand. Hence the points $(0.523599,-1)$ and $(0.523599+2\pi/3,1)$ will be included in the list of data points used for the interpolation of $G$. Note that this corresponds to a bijection $w_{1}$ that sends the red strand to $-1$ and the green strand to $1$.

Similarly, we compute the intersection points of the other pairs of strands and obtain the list of data points $(0.523599,-1)$, $(0.912415,1)$, $(0.134782+2\pi/3,-1)$, $(0.523599+2\pi/3,1)$, $(1.15567+2\pi/3,1)$, $(1.5708+2\pi/3,-1)$, $(1.98592+2\pi/3,1)$, $(0.134782+4\pi/3,1)$, $(0.912415+4\pi/3,-1)$, $(1.15567+4\pi/3,-1)$, $(1.5708+4\pi/3,1)$, $(1.98592+4\pi/3,-1)$.

Performing a trigonometric interpolation through these points yields 
\begin{align}
G(t)&=19.0248-0.823358\cos(t)+17.1048\sin(t)-15.2722\cos(2t)-0.13139\sin(2t)\nonumber\\
&\ \ -0.454434\cos(3t)-12.8637\sin(3t)-0.823379 \cos(4t)+8.6227\sin(4t)\nonumber\\
&\ \ -4.10823\cos(5t)-0.818417\sin(5t).
\end{align}

Using $F$ and $G$ in Eq. (\ref{eq:fourierpara}) results in a Fourier parametrisation of the desired braid, closing to $5_{2}$. This allows us to construct $g_{a,b}$ and then $f_{a,b}$.
In \cite{us:2016fivetwo} we describe the construction of a semiholomorphic polynomial with zero level set of the form of the knot $5_{2}$. This involves finding a Fourier parametrisation for its braid through geometric considerations alone. The parametrisation found by our algorithm is significantly more complicated than the one described in \cite{us:2016fivetwo}. However, the procedure in \cite{us:2016fivetwo} is not algorithmic and in general very time consuming.

As discussed in \cite{us:2016fivetwo} braids whose words are identical up to signs, i..e $B_{1}=\sigma_{i_{1}}^{\varepsilon_{i_{1,1}}}\sigma_{i_{2}}^{\varepsilon_{i_{2,1}}}\ldots\sigma_{i_{\ell}}^{\varepsilon_{i_{\ell,1}}}$ and $B_{2}=\sigma_{i_{1}}^{\varepsilon_{i_{1,2}}}\sigma_{i_{2}}^{\varepsilon_{i_{2,2}}}\ldots\sigma_{i_{\ell}}^{\varepsilon_{i_{\ell,2}}}$ with all $\varepsilon_{i}\in\{\pm1\}$, can be parametrised using the same $F_{C}$. For example the function $F$ computed above can also be used to parametrise the $x$-coordinate of the braid $\sigma_{1}\sigma_{2}^{-1}\sigma_{1}\sigma_{2}^{-3}$, which closes to the knot $6_{2}$.

In \cite{us:2016fivetwo} we use geometric inspection to find a finite Fourier parametrisation for the braid word $\sigma_{1}^{-1}\sigma_{2}\sigma_{1}^{3}\sigma_{2}$, which is obviously also a minimal braid word for the knot $5_{2}$, since it differs from the one given above only by conjugation by $\sigma_{2}$. There are some fundamental differences between the approach in \cite{us:2016fivetwo} and the algorithm described here. 
The method from \cite{us:2016fivetwo} requires a lot of testing. Choosing a braid representative which is as symmetric as possible typically allows us to parametrise the braid using only few terms and low spatial frequencies. Geometric considerations give us a rough idea of a possible range of coefficients and spatial frequencies, so that it becomes feasible to find a finite Fourier parametrisation for the given braid.
The desire for symmetry makes $\sigma_{1}^{-1}\sigma_{2}\sigma_{1}^{3}\sigma_{2}$ a better choice of braid word, in particular if we place the crossings in equal distances along the $t$-axis and have the first crossing at $t=0$.
Note that this is different from the positions of the crossings in the description of our algorithm. In order to calculate the trigonometric interpolation efficiently by using the discrete Fourier transform, we want the data points to be equally distributed along the $t$-axis starting at $t=0$. Since the data points should lie between crossings, we explicitly do not want a crossing at $t=0$.

For short and symmetric braids with few strands the method from \cite{us:2016fivetwo} might lead to simpler polynomials in the sense of fewer terms and lower orders, for example the parametrisation for the knot $5_{2}$ in \cite{us:2016fivetwo} only has two terms for $F$ and $G$ each and the highest order is seven. However, since the method is not algorithmic, it is in general not a good way of finding parametrisations for more complicated braids.

With the described algorithm we can find a finite Fourier parametrisation for any braid $B$. Then $g_{a,b}$ defined as in Eq. (\ref{eq:braidpoly}) can be written as a polynomial in $u$, $\rme^{\rmi t}$ and $\rme^{-\rmi t}$ by Lemma \ref{poly}. Constructing $f_{a,b}$ as in Section \ref{fourier} gives a semiholomorphic polynomial which by the results from Section \ref{proof} has the closure of $B$ as its zero level set on the three-sphere of radius one if $a$ and $b$ are small enough. This makes our proof of the existence of semiholomorphic polynomials with knotted zero level sets fully algorithmic.

\section{Properties of the constructed polynomials}
\label{Props}
In this section we prove some properties of the constructed polynomials. We use the notation from Section \ref{main}, in particular the symbol $f_{\lambda}$ will always denote a polynomial that was constructed using the method described in Section \ref{main}.
\begin{cor}
\label{transverse}
If $\lambda$ is small enough, $0$ is a regular value of $f_{\lambda}|_{S^{3}}$.
\end{cor}
\begin{proof} Let $(u_{*},v_{*})\in S^{3}$ be a point on the link, so $f_{\lambda}(u_{*},v_{*})=0$. Then for fixed $v=v_{*}$, the function $f_{\lambda}$ is a polynomial in $u$ with $u_{*}$ being a simple root. Hence $\frac{\partial f_{\lambda}}{\partial u}(u_{*},v_{*})\neq 0$, so in particular $\nabla f_{\lambda}(u_{*},v_{*})$ has full rank.\newline
Furthermore, for small fixed $\lambda$ a straightforward calculation shows that the intersections of the curves $(u_{\lambda,j}(r,t),r\rme^{\rmi t})$, $j=1,2,\ldots,s$, $t\in[0,2\pi]$ fixed, with $S^{3}$ are transverse. Hence $(u_{*},v_{*})$ is a regular point of $f|_{S^{3}}$.

\end{proof}
This corollary is important for applications which focus on functions $\mathbb{R}^{3}\rightarrow\mathbb{C}$. In order to make the functions satisfy extra physical conditions, which depend on the individual setting, coefficients might have to be slightly changed. The fact that $0$ is a regular value offers a certain stability under small perturbations of the originally constructed polynomial. This means that if we do not change the coefficients too much, the zero set of interest will remain the desired link.

In \cite{rudolph:2005knot} Lee Rudolph discusses transverse  $\mathbb{C}$-links. These are links which can be realised as the transverse intersection of the zero level set of a complex polynomial $f:\mathbb{C}^{2}\to\mathbb{C}$ with the unit three-sphere. A link is a transverse $\mathbb{C}$-link if and only if it is quasi-positive. Corollary \ref{transverse} shows that if we relax the condition from $f$ being a complex polynomial to being semiholomorphic, then every link can be realised as the transverse intersection of the zero level set of such a polynomial with $S^{3}$.

\begin{cor}The polynomial $f_{\lambda}$ is harmonic, that is $(\partial_{u}\partial_{\overline{u}}+\partial_{v}\partial_{\overline{v}})f_{\lambda}=0$.
\end{cor}
\begin{proof}This is immediate from the construction. Since $f_{\lambda}$ is a polynomial in $u$, there is no dependence on $\overline{u}$. Recall that $f_{\lambda}$ was obtained from $g_{a,b}$ by replacing $\rme^{\rmi t}$ by $v$ and $\rme^{-\rmi t}$ by $\overline{v}$ and $g_{a,b}$ is a polynomial in $u$, $\rme^{\rmi t}$ and $\rme^{-\rmi t}$. Now suppose there was a monomial containing both $v$ and $\overline{v}$, say $v^n \overline{v}^m$, as a factor. Then in the polynomial expression of $g_{a,b}$ these would have simplified to $\rme^{\rmi t(n-m)}$. Thus $\partial_{\overline{v}}f_{\lambda}$ does not depend on $v$ which proves the corollary.
\end{proof}

Not every polynomial of the form constructed in Section \ref{main} will be of practical use. In particular with regard to the engineering of physical knotted fields, a polynomial with too many terms or of too high degree might be too hard to control to be applicable to some of the systems described in \cite{us:2016lemniscate}. This is the reason why a naive algebraic approximation of a link given as a parametric curve is not particularly useful.

The advantage of our construction is that it allows for an upper bound of the degree of the constructed polynomial $f_{\lambda}$ in terms of the braid data.

We use the notation from Section \ref{main}. The algorithm in Section \ref{algo} finds a finite Fourier parametrisation (\ref{eq:fourierpara}) of a braid $B$ that closes to the desired link $L$. Using trigonometric interpolation to find such a parametrisation allows us to give a bound on the highest order $D=\max_{C\in\mathcal{C}}\{N_{C},M_{C}\}$ in terms of the length $\ell$ of the braid and the number of strands $s_{C}$ in each link component $C$.

\begin{lemma}Let $\mathcal{C}$ be again the set of all link components of $L$ and let $\tilde{C}$ be the link component which consists of the most strands of the braid $B$ which is used to construct $f_{\lambda}$. Then 
\begin{equation}D\leq\max\left\{\left\lfloor\frac{s_{\tilde{C}}\ell-1}{2}\right\rfloor,\left\lfloor\frac{\ell s_{\tilde{C}}^{2}|\mathcal{C}|+s_{\tilde{C}}(1-\ell)-2}{2}\right\rfloor\right\},\end{equation} 
where $s_{\tilde{C}}$ is the number of strands which $\tilde{C}$ consists of and $\ell$ is the length of the braid word.
\end{lemma}

\begin{proof}
The degrees of the trigonometric polynomials $F_{C}$ and $G_{C}$ can be directly calculated from the number of data points used for the interpolation.
Note that the number of data points needed for the trigonometric interpolation of each of the polynomials $F_{C}$ is equal to $s_{C}\ell$. Thus $N_{C}$ is equal to $\lfloor\frac{s_{C}\ell-1}{2}\rfloor$, where $\lfloor x\rfloor$ denotes the largest integer less than or equal to $x$.

The union of the graphs of $F_{C}\left(\frac{h+2\pi j}{s_{C}}\right)$ form a braid diagram $B'$ with unspecified signs of crossings. We can choose the signs in such a way that the resulting braid is identical to the input braid $B$, although the braid word might be different. The number of crossings of $B'$ is equal to the length $\ell'$ of its braid word and hence by the discussion in Section \ref{algo} at least $\ell$.
    
For the trigonometric interpolation of one trigonometric polynomial $G_{C}$ we need a data point for every crossing that involves a strand of $C$, counting with multiplicities. The crossings of $B'$ correspond to intersections of certain trigonometric polynomials related to the different $F_{C}$. Since the trigonometric polynomials can be be associated with complex polynomials on the unit circles, by the fundamental theorem of algebra their number of intersections can be bounded in terms of their degree.

Let $C\in\mathcal{C}$ be a component consisting of $s_{C}$ strands. We first consider the number of intersections between different strands of the same component $C$. Recall that for a trigonometric polynomial $F_{C}(h)=\sum_{k=-N_{C}}^{N_{C}}c_{k,C}\rme^{\rmi kt}$ we can construct a complex polynomial $p_{C}(z)$ of degree $2N_{C}$ with $p_{C}(\rme^{\rmi t})=F_{C}(t)$ by defining $p(z)=z^{N_{C}}\sum_{k=-N_{C}}c_{k,C}z^{k}$. The sum of intersections between strands with index $j$ and $j+1\mod s_{C}$ are exactly points on the unit circle where the complex polynomials $p_{C}(z)$ and $p_{C}(z\rme^{\frac{\rmi 2\pi}{s_{C}}})$ are equal. Since both have degree $2N_{C}$, this number is at most $2N_{C}\leq s_{C}\ell-1$. In general the sum of intersections between strands with index $j$ and $j+k$ correspond to points on the unit circle where $p_{C}(z)$ and $p_{C}(z\rme^{\frac{\rmi 2\pi k}{s_{C}}})$ coincide. Thus there are again at most $2N_{C}\leq s_{C}\ell-1$ many of these. In order to capture all pairs of strands in the component $C$, we have to sum over $k$ from 1 to $\lfloor\frac{s_{C}+1}{2}\rfloor$. Thus there are at most $\lfloor\frac{s_{C}+1}{2}\rfloor(s_{C}\ell-1)$ many intersection points between two different strands of the same component $C$. Hence the number of data points needed to achieve the correct signs for these crossings are at most $2\lfloor\frac{s_{C}+1}{2}\rfloor(s_{C}\ell-1)\leq(s_{C}+1)(s_{C}\ell-1)$.

Intersections involving two different link components, especially if they consist of different numbers of strands, are more complicated to count. The points $t\in[0,2\pi]$ with $F_{C}\left(\frac{t+2\pi j}{s_{C}}\right)=F_{C'}\left(\frac{t+2\pi k}{s_{C'}}\right)$ correspond to points $z=\rme^{it}$ on the unit circle, where \begin{equation}p_{C}(z^{\lcm(s_{C},s_{C'})/s_{C}})=p_{C'}(z^{\lcm(s_{C},s_{C'})/s_{C'}}\rme^{\frac{m2\pi \rmi}{\gcd(s_{c},s_{C'})}})\end{equation}
for some $m\in\mathbb{Z}$. Here $\lcm(s_{C},s_{C'})$ denotes the least common multiple of $s_{C}$ and $s_{C'}$ and $\gcd(s_{C},s_{C'})$ is their greatest common divisor. Again we can use the degrees of the polynomials, which are both $2N_{C}\lcm(s_{C},s_{C'})/s_{c}\leq \lcm(s_{C},s_{C'})l$, to bound the number of intersections points. To capture all possible pairs of strands, we have to sum over $m$ from $1$ to $\gcd(s_{C},s_{C'})$. Thus in total there are at most $\lcm(s_{C},s_{C'})\ell\  \gcd(s_{C},s_{C'})=\ell s_{C}s_{C'}$ intersection points where one of the strands is in $C$ and the other $C'$. Hence if $\tilde{C}$ is the component which consists of the most strands, the number of intersection points where exactly one of the strands is from $C$ is at most $\ell s_{\tilde{C}}^{2}(|\mathcal{C}|-1)$ which equals the number of data points for the trigonometric interpolation for $G_{C}$ to achieve the correct signs for these crossings.

Thus the total number of data points needed for $G_{C}$ is at most $\ell s_{\tilde{C}}^{2}(|\mathcal{C}|-1)+(s_{C}+1)(s_{C}\ell-1)$. This is the greatest if $C=\tilde{C}$ and hence the degree of $G_{C}$ is at most 
\begin{equation}M_{C}\leq\left\lfloor\frac{\ell s_{\tilde{C}}^{2}(|\mathcal{C}|-1)+(s_{\tilde{C}}+1)(s_{\tilde{C}}\ell-1)-1}{2}\right\rfloor=\left\lfloor\frac{\ell s_{\tilde{C}}^{2}|\mathcal{C}|+s_{\tilde{C}}(\ell-1)-2}{2}\right\rfloor.\end{equation}
Thus 
\begin{equation}D=\max\{N_{C},M_{C}\}\leq\max\left\{\left\lfloor\frac{s_{\tilde{C}}\ell-1}{2}\right\rfloor,\left\lfloor\frac{\ell s_{\tilde{C}}^{2}|\mathcal{C}|+s_{\tilde{C}}(\ell-1)-2}{2}\right\rfloor\right\},\end{equation} 
where again $\tilde{C}$ is the component with the most strands. 
\end{proof}

\begin{lemma}
The degree of $f_{\lambda}$ is equal to $\max\{D,s\}=\max_{C\in\mathcal{C}}\{N_{C},M_{C},s\}$.
\end{lemma}

\begin{proof}The degree of $f_{\lambda}$ with respect to $u$ is $s$.
Note that the total degree of a monomial of $f_{\lambda}$ for which the degree with respect to $u$ is $k$ is $d\leq\frac{D}{s}(s-k)+k$. If $D\geq s$, then $d\leq D+k(1-\frac{D}{s})\leq D$. If $D<s$, then $d\leq (s-k)+k=s$.
\end{proof}

Using the bound we have for $D$, we get: 
\begin{cor}
The degree of $f_{\lambda}$ is bounded above by 
\begin{equation}
\label{eq:upper}
\deg(f_{\lambda})\leq c_{2}:=\max\left\{\left\lfloor\frac{s_{\tilde{C}}\ell-1}{2}\right\rfloor,\left\lfloor\frac{\ell s_{\tilde{C}}^{2}|\mathcal{C}|+s_{\tilde{C}}(\ell-1)-2}{2}\right\rfloor,s\right\}.\end{equation}
\label{bound}
\end{cor}

Note that for knots $|\mathcal{C}|=1$ and hence for non-trivial knots $\deg(f_{\lambda})\leq\lfloor\frac{\ell s^{2}+s(\ell-1)-2}{2}\rfloor$, since in this case $s\geq 2$ and $\ell\geq 3$.

Also note that the bound given in Corollary \ref{bound} holds for all semiholomorphic polynomials constructed using the algorithm described in Section \ref{main} and \ref{algo}, in particular using trigonometric interpolation to find the trigonometric braid parametrisation as described in Section \ref{algo}. Corollary \ref{bound} is not a statement about the non-existence of polynomials of a certain degree.

We can also give a lower bound for the polynomial degree, which holds for all semiholomorphic polynomials constructed as in Section \ref{main}, whether trigonometric interpolation as in Section \ref{algo} is used or not.
For any braid parametrisation of the form (\ref{eq:fourierpara}) the degree of the trigonometric polynomials is bounded below in terms  of the number of crossings between pairs of strands. Again we can write the trigonometric polynomials as complex polynomials restricted to the unit circle such that crossings of strands correspond to points on the unit circle where the two corresponding polynomials share the same value.
As before we have to sum over all possible pairs of strands and obtain the following bound.

\begin{cor}
Let $C'$ be the component of the braid $B$ such that the degree $N_{C'}$ of the trigonometric polynomial $F_{C'}$ used to parametrise $B$ as in Eq. (\ref{eq:fourierpara}) is $\max\{N_{C}:C\in\mathcal{C}\}$. Then
\begin{equation}
\label{eq:lower}
N_{C'}\geq c_{1}:=\max\left\{\max_{C\neq C'}\left\{\frac{k_{C}}{2\max\{s_{C},s_{C'}\}}\right\},\frac{2k}{2s_{C'}-1}\right\},
\end{equation}
where $k_{C}$ denotes the number of crossings, where one strand is in the component $C$ and the other in $C'$, and $k$ denotes the number of crossings, where both strands are in $C'$. As before $s_{C}$ denotes the number of strands in the component $C$.
\label{lowbound}
\end{cor} 
\begin{proof}We only give a sketch of the proof here, since it is the same principle as the proof of Corollary \ref{bound}.
In order for the $F_{C}$ to provide a parametrisation of the $x$-coordinate of the braid $B$ as in Eq. (\ref{eq:fourierpara}), each pair of strands has to cross at least a prescribed number of times. The values of $t\in[0,2\pi]$ where these crossings occur, correspond to points on the unit circle where two complex polynomials agree. This yields a lower bound for the degrees of these polynomials which are easily related to the different $N_{C}$.\end{proof}

If the braid closes to a knot, there is only one component, so $N_{C'}\geq \frac{2\ell}{2s-1}$, where $\ell$ is the length of the braid word.
Corollary \ref{lowbound} implies that $\deg(f_{\lambda})\geq \max\{s,c_{1}\}$.
From the polynomials we have constructed, we find that the bounds given by Corollary \ref{lowbound} and Corollary \ref{bound} are not very tight bounds. The proofs can explain this, since the degree was determined by the number of data points which in turn was determined by the number of points on the unit circle where two complex polynomials agree. This number was bounded by the degree of these complex polynomials, but of course in general a complex polynomial does not have all of its roots on the unit circle.
The results of Corollaries (\ref{transverse})-(\ref{lowbound}) are summarised in Theorem \ref{thmintro2} in Section \ref{intro}.

We have proven the existence of a semiholomorphic polynomial $f$ of bounded degree, whose zero level set on the unit three-sphere is a given link.
Applying the standard stereographic projection
\begin{equation}
u=\frac{x^2+y^2+z^2-1+2\rmi z}{x^2+y^2+z^2+1}, \quad v=\frac{2(x+\rmi y)}{x^2+y^2+z^2+1}
\label{eq:stereo}
\end{equation}
to $f$ results in rational function, whose denominator is a constant times some power of $(x^2+y^2+z^2+1)$. Hence multiplying by the common denominator yields a polynomial $\mathbb{R}^3\rightarrow\mathbb{C}$ in $x$, $y$ and $z$, whose zero level set is $L$.
It follows from Lemma \ref{transverse} that the coefficients of this polynomial can be taken to be Gaussian integers. This shows:

\begin{cor}
Let $B$ be a braid on $s$ strands of length $\ell$ and let $L$ denote its closure. Then there exist $F_{1}, F_{2}\in\mathbb{Z}[x,y,z]$ such that the vanishing set of $(F_{1},F_{2})$ over the reals $\{(x,y,z)\in\mathbb{R}^{3}:F_{1}(x,y,z)=F_{2}(x,y,z)=0\}$ is ambient isotopic to $L$. Furthermore 
\begin{equation}
\max\{\deg(F_{1}),\deg(F_{2})\}\leq\max\left\{2\left\lfloor\frac{s_{\tilde{C}}\ell-1}{2}\right\rfloor,2\left\lfloor\frac{\ell s_{\tilde{C}}^{2}|\mathcal{C}|+s_{\tilde{C}}(\ell-1)-2}{2}\right\rfloor,2s\right\},
\end{equation}
where $s_{\tilde{C}}$ is the number of strands of the link component $\tilde{C}\in\mathcal{C}$ which consists of the most strands.
\end{cor}

In \cite{us:2016lemniscate} we described the construction of semiholomorphic polynomials for the family of lemniscate knots, a generalisation of torus knots, and discussed applications of these functions to physical systems. We will not go through the details here, but we would like to point out that all results about applications still hold. In particular, for any link $L$ which is the closure of a braid on $s$ strands and any $N\in\mathbb{Z}$ we can construct a rational map $W:S^{3}\rightarrow S^{2}$ of topological degree $Ns$ such that $W^{-1}(0,0,-1)=L$.
Similarly for any link $L$ which is the closure of a braid on $s$ strands and any $N\in\mathbb{Z}$ we can construct a vector field $V:\mathbb{R}^{3}\rightarrow\mathbb{R}^{3}$ such that some flow lines of $V$ form the link $L$ and the helicity of $V$ is equal to $Ns$.
This allows us to implement any link type as an initial condition in the wide range of physical systems mentioned in Section \ref{intro}.

Another result that was shown for lemniscate knots in \cite{us:2016lemniscate} remains true in this most general setting without changing the proof given in \cite{us:2016lemniscate}.

\begin{lemma}\label{crit}
\cite{us:2016lemniscate} If $g_{a,b}$ has exactly $n$ non-degenerate phase-critical points, i.e. $\nabla_{\mathbb{C}\times S^{1}} \arg(g_{a,b}(x))$ has full rank for all $x\in\mathbb{C}\times S^{1}$ except on $n$ points and $a$ and $b$ are small, then $f_{a,b}|_{S^{3}}$ also has exactly $n$ phase-critical points. It follows that if $n=0$, then $\arg(f_{a,b}):S^{3}\backslash L\rightarrow S^{1}$ is a fibration of $S^{3}\backslash L$ over $S^1$.
\end{lemma}

We call a braid on $s$ strands strictly homogeneous if it is homogeneous and every generator $\sigma_{1},\ldots,\sigma_{s-1}$ appears at least once in its braid word, either with a positive or a negative sign.
Lemma \ref{crit} can be used to show the following result.

\begin{thm}
\label{fibration} Let $B$ be a strictly homogeneous braid on $s$ strands and $L$ be its closure. Then $L$ is fibred and there exists a semiholomorphic polynomial $f:\mathbb{C}^{2}\to \mathbb{C}$ s.t. $\deg_{u} f=s$, $f^{-1}(0)\cap S^{3}=L$  and $\arg(f)$ is a fibration of $S^{3}\backslash L$ over $S^1$.
\end{thm}
\begin{proof} By Lemma \ref{crit} it is enough to show that $B$ has a Fourier parametrisation such that $g_{a,b}$ does not have any phase-critical points. Then for any $\lambda>0$ the function $g_{\lambda a,\lambda b}$ will not have any phase-critical points either. It follows from Theorem \ref{thmintro} that if we choose $\lambda$ small enough, the resulting $f_{a,b}$ will have a zero level set on $S^{3}$ that is ambient isotopic to $L$.
 
First note that the phase-critical points of $g_{a,b}$ are exactly those points  in $\mathbb{C}\times S^1$, where $\partial_u g_{a,b}=0$ and $\partial_t \arg(g_{a,b})=0$. 
Now if $\bigcup_{j=0}^{s-1}(X_{j}(t),Y_{j}(t),t)$ is any parametrisation of $B$ (not necessarily of the form of Eq. (\ref{eq:fourierpara})), we can define $g_{a,b}$ as in Eq. (\ref{eq:braidpoly}). We assume that that $g_{a,b}$ has $s-1$ disjoint non-zero critical values for every $t$. Then the critical values together with a strand $(0,0,t)$, $t\in[0,2\pi]$ form a braid in $\mathbb{C}\times S^1$ on $s$ strands. We call this the critical braid.

Note that the set $W$ of monic polynomials of degree $s$ with disjoint roots and disjoint critical values fibers over $V=\{(z_{1},z_{2},\ldots,z_{s-1})\in\mathbb{C}^{s-1}|z_{i}\neq 0 \text{ and } z_{i}\neq z_{j} \text{ if }i\neq j\}/S_{s-1}$ via the map $\theta$ that sends each such polynomial to its set of critical values \cite{bcn:2002critical}. Let $H$ be a homotopy of loops in $S^{1}\times V$ such that $H$ starts at a loop $x\in S^{1}\times V$ which is in the image of $(\text{id},\theta):S^{1}\times W\rightarrow S^{1}\times V$, where $\text{id}$ is the identity map. Say $x=(\text{id},\theta)(y)$ for some $y\in S^{1}\times W$. Then $H$ lifts to a homotopy of loops in $S^{1}\times W$ starting at $y$. Thus if two critical braids $B_{1}$ and $B_{2}$ are braid isotopic and $p_{t,1}$ is a 1-parameter ($t\in S^{1}$) family of monic polynomials with disjoint roots and critical values such that $(\text{id},\theta)(p_{t,1})=B_{1}$, then the zeros of $p_{t,1}$ form a closed braid in $S^{1}\times\mathbb{C}$ which is isotopic to the braid formed by the roots of some other family of polynomials $p_{t,2}$ with $(\text{id},\theta)(p_{t,2})=B_{2}$.

Let $B=\sigma_{i_{1}}^{\varepsilon_{i_{1}}}\sigma_{i_{2}}^{\varepsilon_{i_{2}}}\ldots\sigma_{i_{\ell}}^{\epsilon_{i_{\ell}}}$, $\varepsilon_{i}\in\{\pm1\}$. Then there exists a parametrisation $P$ of $B$ such that if $g_{a,b}$ is defined using $P$, then the critical braid of $g_{a,b}$ can be isotoped to the braid $A_{i_{1},\varepsilon_{i_{1}}}A_{i_{2},\varepsilon_{i_{2}}}\ldots A_{i_{\ell},\varepsilon_{i_{\ell}}}$, where $A_{i,\varepsilon}=(\sigma_{1}\sigma_{2}\ldots\sigma_{i-1})\sigma_{i}^{2\varepsilon}(\sigma_{1}\sigma_{2}\ldots\sigma_{i-1})^{-1}$ \cite{rudolph:2001some}.

For a strictly homogeneous braid $B$, we can write down a parametrisation $\bigcup_{j=0}^{s-1}(X'_{j}(t),Y'_{j}(t),t)$ of the braid of critical values such that $\partial_t \arg(X'_{j}+iY'_{j})>0$ for all $t$ if $\varepsilon_{j}$ is positive, $\partial_t \arg(X'_{j}+iY'_{j})<0$ for all $t$ if $\varepsilon_{j}$ is negative and such that this critical braid is braid isotopic to the critical braid derived from the parametrisation $P$. 
Hence the paramesiation $\bigcup_{j=0}^{s-1}(X'_{j}(t),Y'_{j}(t),t)$ of the critical braid lifts to a parametrisation $(X_{j}(t),Y_{j}(t),t)$ of a braid which is isotopic to $B$ and the function $g_{a,b}$ corresponding to this parametrisation does not have any phase-critical points. 
Since the roots are disjoint and the critical points are disjoint, the dependence of the critical values on the coefficients is in fact differentiable, so a small perturbation of the parametrisation will not introduce any more phase-critical points.
Since trigonometric polynomials are dense in the set of continuous, $2\pi$-periodic real functions, we can approximate the parametrisation $(X_{j}(t),Y_{j}(t),t)$ by a Fourier parametrisation of the form of Eq. (\ref{eq:fourierpara}) such that the corresponding $g_{a,b}$ does not have any phase-critical points either.
\end{proof}

Stallings \cite{stallings:1978constructions} already showed that closures of homogeneous braids are fibred. Our theorem specifies this by providing a certain form of the fibration map.
This proof shows in fact more. Let $B=\sigma_{i_{1}}^{\varepsilon_{i_{1}}}\sigma_{i_{2}}^{\varepsilon_{i_{2}}}\ldots\sigma_{i_{\ell}}^{\varepsilon_{i_{\ell}}}$ and let $\beta(B)$ denote 
\begin{align}
\beta(B)=&\sum_{i=1}^{s-1}\left|\{j\in\{1,2,\ldots,\ell-1\}:\exists\ k\in\{1,2,\ldots,\ell-1\} \text{ s.t. } i_{j}=i_{j+k\text{ mod }\ell}=i,\right.\nonumber\\ 
&\left.i_{j+m\text{ mod }\ell}\neq i \text{ for all }m<k \text{ and }\varepsilon_{i_{j}}\varepsilon_{i_{j+k\text{ mod }\ell}}=-1\}\right|\nonumber\\
&+\left|\{j\in\{1,2,\ldots,s-1:\text{ There is no }k\text{ s.t. }i_{k}=j\}\}\right|.\end{align}
Note that for strictly homogeneous braids, it is $\beta(B)=0$. The $\beta$-value measures how far a braid is from being strictly homogeneous.
Together with Lemma \ref{crit} the proof of Theorem \ref{fibration} shows that for any braid $B$ on $s$ strands, there exists a semiholomorphic polynomial $f:\mathbb{C}^{2}\rightarrow\mathbb{C}$ such that $f^{-1}(0)\cap S^{3}$ is the closure of $B$, $\deg_{u}f=s$ and $\arg(f)$ has exactly $\beta(B)$ critical points. 

It can be arranged that $\arg(f)$ is a smooth circle-valued Morse function on $S^{3}\backslash L$. Such a function is called \emph{regular} if there is a diffeomorphism $\varphi$ of the union of $|\mathcal{C}|$ solid tori, such that the composition of $\varphi$ and $\arg(f)$ applied to a tubular neighbourhood of the link is just the projection map $L\times(\mathring{D}\backslash\{0\}) \rightarrow S^{1}:(x,y)\mapsto y/|y|$.
Since $\arg(f)$ is the argument of a semiholomorphic polynomial, it is regular.

The Morse-Novikov number $\mathcal{MN}(L)$ of a link $L$ is the minimal number of critical points of all smooth, regular circle-valued Morse functions \cite{prw:mn}. Thus $\mathcal{MN}(L)=0$ if and only if $L$ is fibred.

The discussion above then implies 
\begin{equation}
\mathcal{MN}(L)\leq \beta(B).
\end{equation} 
We should point out that this bound is a strict inequality for some knots. For example the knot $8_{20}$ is known to be fibered and has thus $\mathcal{MN}(L)$ equal to zero. However, it is not the closure of a homogeneous braid and hence $\beta(B)>0$ for all braids $B$ that close to $8_{20}$ \cite{bell:2012homogeneous}.

The proof of  Theorem \ref{fibration} follows ideas from \cite{rudolph:2001some} very closely, so we believe that we might not be the first to derive this bound. What we have shown here is that the bound can be realised by the argument of a semiholomorphic polynomial.

\section*{Acknowledgments}
We are grateful to Gareth Alexander, Mark Bell, David Foster, Daniel Peralta-Salas and Jonathan Robbins for discussions and comments.
This work is funded by the Leverhulme Trust Research Programme Grant RP2013- K-009, SPOCK: Scientific Properties Of Complex Knots.

\end{document}